\documentclass[smallextended]{svjour_edited}       % onecolumn (second format)
\usepackage{graphicx}
\usepackage{graphicx,epsfig}
\usepackage{amsfonts}
\usepackage{amssymb}
\usepackage{amsmath}
\usepackage{color}

\renewcommand{\vec}[1]{ \mathbf{#1}}
\newcommand{\mat}[1]{ {#1}}
\renewcommand{\Re}{\mbox{\it Re} \;}
\renewcommand{\Im}{\mbox{\it Im} \;}
\renewcommand{\exp}[1]{e^{#1}}

\begin{document}

\title{Field of values analysis of preconditioners for the Helmholtz equation in lossy media \thanks{The author was supported by the grant 133174 from the Academy of Finland}
}

\titlerunning{FOV analysis of preconditioners for the Helmholtz equation}        % if too long for running head

\author{Antti Hannukainen}

\institute{A. Hannukainen \at
              Aalto University, Department of Mathematics and Systems Analysis, P.O. Box 11100, FI-00076 Aalto, Finland  \\
              Tel.: +358 9 470 23696 \\
              Fax: +358 9 470 23016  \\
              \email{antti.hannukainen@hut.fi}}

\date{\today}

\maketitle
 
\begin{abstract}
In this paper, we analyze the convergence of the preconditioned GMRES method for the first order finite element discretizations of the Helmholtz equation in media with losses. We consider a Laplace preconditioner, an inexact Laplace preconditioner and a two-level preconditioner. Our analysis is based on bounding the field of values of the preconditioned system matrix in the complex plane. The analysis takes the non-normal nature of the linear system naturally into account and allows us to easily consider certain type of inexact Laplace preconditioners via a perturbation argument.  For the two-level preconditioner, our convergence analysis takes into account a media, which has not been considered in previous works.
\end{abstract}

\section{Introduction}
\label{intro}

Efficiently solving the Helmholtz equation is one of the major challenges in the field of numerical analysis. Current methods have difficulties both with the discretization and with the solution of the resulting linear system when the frequency grows. The first of these two difficulties is due to the very large number of elements required to obtain a meaningful numerical approximation to an highly-oscillating function. For example, the analysis in \cite{BaIh:95,BaIh:97} states that for the first order finite element method, the mesh size $h$ should satisfy the bound $\kappa^2 h \ll 1$ before the asymptotic convergence rate is achieved for a system with first order absorbing boundary conditions. Here $\kappa$ is the wavenumber related to frequency of the resolved problem. In practice, satisfying such requirement for large wavenumbers $\kappa$ leads to solving very large linear systems.

Developing efficient solvers for the linear systems arising form discretizations of the Helmholtz equation has proven to be considerably more difficult than for elliptic problems. This is due to the indefinite nature of the linear system. In addition, if absorbing boundary conditions or certain type of losses are included, the discretization matrix is also complex valued and non-normal. This and the indefinite nature of the problem render many successful solution methods for positive definite problems less useful for the Helmholtz equation. For example, when applied to linear systems arising from the Helmholtz equation, multigrid methods suffer from problems both in the smoothing and in the coarse grid correction steps, see e.g. \cite{El:01}. To have a convergent method, the coarse grid has to satisfy the same density constraints as the original discretization. In practice this means that although working multigrid preconditioners have been developed, they only provide a small benefit over a direct solver.

The current trend for solving the linear system is to use a preconditioner together with a suitable Krylov subspace solver. The indefinite matrix problem can be solved with several Krylov subspace methods, e.g. GMRES, BiCGStab, etc. (see \cite{Gr:1997,Sa:2003}). From these methods, a covergence theory exists only for the GMRES method (see e.g \cite{Gr:1997}). Because of this, we consider different preconditioners in connection with the GMRES method.

The existing preconditioners can be divided into two groups, shifted-Laplace preconditioners (see e.g \cite{ErOoVu:04,ErOoVu:06,Er:08,ErOoVu:06b}) and two-level preconditioners (see e.g. \cite{BrPaLe:1993,CaWi:1993,CaWi:92}). The shifted-Laplace preconditioners are further development of Laplace preconditioners, which have been studied by several authors, e.g. \cite{Ys:1989,BaGoTu:1983}.  These methods are successful in cutting the growth in the condition number due to the Laplace operator part. However, a $\kappa$-dependency in the required number of iterations still remains for the preconditioned system. Based on numerical examples \cite{ErOoVu:04,ErOoVu:06},  the number of iterations for the shifted-Laplace type preconditioners for problems with Dirichlet boundary conditions behaves like  $O(\kappa^2)$ and like $O(\kappa)$ for problems with absorbing type boundary conditions. Regardless of this asymptotic behavior, the introduction of the shift-term leads to a lower number of iterations compared to pure Laplace preconditioners. As the iterative methods for solving indefinite systems are computationally costly and memory intensive, even a small reduction in the number of required iterations is important. 

The two-level preconditioners are based on combining a Laplace preconditioner with a coarse grid correction. These methods can deliver $\kappa$-independent number of iterations, but they suffer from identical problems as multigrid methods. Namely, a direct solver has to be employed to compute the coarse grid correction on a mesh satisfying the same constraints with the original discretization. The analysis of such methods has been performed for real valued Helmholtz equation in \cite{BrPaLe:1993,CaWi:1993} by using tools from the analysis of additive Schwarz methods for elliptic problems.

In this paper, we will develop a field of values (FOV) based method to analyze the convergence of the preconditioned GMRES for the finite element discretizations of the Helmholtz equation with homogeneous Dirichlet boundary conditions in lossy media. We will consider a Laplace preconditioner, an inexact Laplace preconditioner and a two-level preconditioner. A similar analysis has been done in \cite{GiEr:06} for Hermitian positive definite split preconditioners using algebraic tools. The main difference to this work is that we estimate the FOV by using methods similar to the ones applied in the analysis of additive Schwarz preconditioners for elliptic problems (see e.g. \cite{ToWi:05}). 

The convergence of GMRES with a Laplace or a shifted-Laplace preconditioner has been previously analyzed in the literature, e.g. \cite{GiErVu:07}, by using algebraic tools. For certain kinds of losses, the preconditioned system matrix is diagonalizable in the inner product induced by a weighted mass matrix. In this approach, the eigenvalues are analyzed and the non-normality is taken into account by considering the conditioning of the weighted mass matrix.

Our analysis takes the non-normal nature of the problem automatically into account and allows us to analyze the inexact Laplace preconditioner via a perturbation argument.  We are also able to give convergence bounds for two-level preconditioner in lossy media. As such a preconditioner is not positive definite nor Hermitian, it is not covered by previous works. We will also give a more detailed analysis of the $\kappa$-dependency of the coarse grid mesh size $H$ when using  two-level preconditioners. Our analysis indicates that in the worst case, the constraint $\kappa^3 H \ll 1$ should be satisfied to guarantee convergence of the two-level method. The same mesh size constraint is also valid for the actual computational grid. The different mesh size requirement in comparison to \cite{BaIh:95,BaIh:97} is due to different boundary conditions.

The organization of this paper is the following. First we introduce the model problem and quickly review the field of values based convergence theory of the GMRES method. We then introduce the preconditioners and give bounds for the FOV in each case.  We conclude the paper with numerical examples on all of the proposed methods.

\section{Preliminaries}

We consider the problem 

\begin{equation}
\begin{aligned}
\label{eq:strong_problem}
\Delta u + (\kappa^2 - \mathrm{i} \sigma) u & = f  \quad \mbox{in }\Omega \\
u & = 0 \quad \mbox{on } \partial \Omega
\end{aligned}
\end{equation}

\noindent where $\kappa \in \mathbb{R}$ and polyhedral domain $\Omega \subset \mathbb{R}^d, d = 2,3$. For the analysis of the exact and inexact Laplace preconditioners we assume that $\sigma \in \mathbb{R}, \sigma > 0$. More general losses, $\sigma \in L^\infty(\Omega)$,

\begin{equation}
\label{eq:general_sigma}
\sigma \geq 0 \quad \mbox{and} \quad  \sigma \geq \sigma_m  > 0 \; \mbox{in} \; \omega \subset \Omega,
\end{equation}

\noindent  are considered in connection with the the two-level preconditioner. The more general case allows the presence of lossless areas, but introduces additional challenges in the analysis. 

%\begin{remark}
%A variable $\kappa$ and $\sigma = \alpha \kappa^2, \alpha \in \mathbb R$ are considered in some previous works on shifted-Laplace preconditioners, e.g. \cite{GiErVu:07}. With a minor modification, our analysis for the exact and inexact Laplace preconditioners applies also to such a case. However, for notational simplicity, we have decided to consider only the case $\kappa \in \mathbb{R}$.
%\end{remark}

The weak form of problem (\ref{eq:strong_problem}) is: Find~$u \in H_0^1(\Omega)$~such that

\begin{equation}
\label{eq:weak_problem}
a(u,v) = (f,v) \quad \forall v \in H^1_0(\Omega),
\end{equation}

\noindent in which $(\cdot,\cdot)$ is the $L^2(\Omega)$ - inner product and

\begin{equation}
\label{eq:sesquilinear_form}
a(u,v) = (\nabla u,\nabla v) -\kappa^2 (u,v) + \mathrm{i} (\sigma u,v).
\end{equation}

\noindent Under our assumptions on $\sigma$, the weak problem (\ref{eq:weak_problem}) has a unique solution. This follows from the unique continuation principle (see e.g. \cite{Le:97}) and the Fredholm alternative. 

In the analysis of the two-level preconditioners, we will use a duality argument. For this purpose, we need to consider the regularity and stability of the solution to problem (\ref{eq:weak_problem}). The coarse grid mesh size requirement $\kappa^3 H \ll 1$ will arise from the $\kappa$-dependency of the stability estimate. The weak solution will have the same regularity as the Poisson problem.

\begin{theorem} 
\label{th:regularity}
Let $f \in L^2(\Omega)$, $\kappa \in \mathbb{R}$, $\sigma \in L^{\infty}(\Omega)$ and $u$ be the weak solution to (\ref{eq:strong_problem}). Then $u \in H^{3/2 + \delta}(\Omega)$ with some $\delta > 0$.

\end{theorem}
\begin{proof} Clearly $u$ is also the weak solution to the problem

\begin{align*}
\Delta u & = (\kappa^2 - \mathrm{i} \sigma) u + f  \quad \mbox{in } \Omega, \\
u & = 0 \quad \mbox{on } \partial \Omega.
\end{align*}

\noindent This is a Poisson problem with the right hand side in $L^2(\Omega)$. Hence, the regularity of $u$ follows directly from the regularity theory for the Poisson equation, see e.g. \cite{Gr:92,Da:88}.
\end{proof}

The parameter $\delta$ in the above theorem is dependent on the shape of the polyhedral domain $\Omega$. For example, in convex domains $\delta = 1/2$. This dependency is analyzed carefully in \cite{Gr:92,Da:88}

In addition to the above regularity result, we need the stability estimate 

\begin{equation*}
\| u \|_{3/2+\delta} \leq C_S \| f\|_0.
\end{equation*}

\noindent Our interest is especially in the $\kappa$-dependency of the constant $C_S$. Studying this dependency for a general $\sigma$ is very difficult. Hence, we will give the estimate only for the case $\sigma \in \mathbb{R}$.

\begin{theorem} 
\label{th:stability}
Let $f \in L^2(\Omega)$, $\kappa \in \mathbb{R}$, $\sigma \in \mathbb{R}$ and let $u$ be the weak solution to (\ref{eq:strong_problem}). Then there exist a constant $C > 0$, independent on $\kappa$ and $\sigma$, such that

\begin{equation*}
\|u \|_{3/2 + \delta} \leq C \left( 1 + \frac{\kappa^2}{\sigma} \right) \| f \|_0.
\end{equation*}

\noindent for some $\delta \in (0,1/2]$.

\end{theorem}
\begin{proof} Let $\{ \varphi_i \}_{i=1}^\infty$ be the eigenfunctions of the Laplace operator, i.e.,

\begin{equation*}
-\Delta \varphi_i = \lambda_i \varphi_i , \quad \mbox{and}\quad (\varphi_i,\varphi_j ) = \delta_{ij}. 
\end{equation*}

\noindent In this basis, the solution to the Helmholtz equation is 

\begin{equation*}
u = \sum_{i =1}^\infty \frac{1}{ \lambda_i - \kappa^2 + \mathrm{i}\sigma } (f,\varphi_i) \varphi_i
\end{equation*}

\noindent The solution satisfies the Poisson problem

\begin{equation*}
-\Delta u = \sum_{i =1}^\infty \frac{\lambda_i}{ \lambda_i - \kappa^2 + \mathrm{i}\sigma } (f,\varphi_i) \varphi_i.
\end{equation*}

\noindent The $L^2(\Omega)$-norm of the right hand side is 

\begin{equation*}
\sqrt{ \sum_{i =1}^\infty  \left( \frac{\lambda_i}{ \left |  \lambda_i - \kappa^2 + \mathrm{i}\sigma \right | } \right)^2 (f,\varphi_i)^2} 
\end{equation*}

\noindent Elementary computations give, 

\begin{equation*}
\frac{x}{ \left | x - \kappa^2 + \mathrm{i}\sigma \right | } \leq \frac{\sqrt{\kappa^4 + \sigma^2} }{\sigma }, \quad \forall x > 0.
\end{equation*}

\noindent By the spectral theory of Laplace opertor, $\lambda_i > 0, i >0$. Hence, we get,

\begin{equation*}
\sqrt{ \sum_{i =1}^\infty  \left( \frac{\lambda_i}{ \left |  \lambda_i - \kappa^2 + \mathrm{i}\sigma \right | } \right)^2 (f,\varphi_i)^2} \leq  \frac{\sqrt{\kappa^4 + \sigma^2} }{\sigma } \sqrt{ \sum_{i =1}^\infty   (f,\varphi_i)^2} \leq C \left( 1 + \frac{\kappa^2}{\sigma } \right) \|f \|_0^2.
\end{equation*}
\noindent Now, the regularity theory for the Poisson problem  yields the desired estimate.
\end{proof}

In the following, we will consider solving the linear system arising from the finite element approximation of the weak problem (\ref{eq:weak_problem}) with first order elements, i.e., the finite element space

\begin{equation}
V_h = \{ \; u \in H^1_0(\Omega) \; | \; v_{| K} \in P_1(K) \; \forall K \; \in \mathcal{T}_h \; \}.
\end{equation}

\noindent The triangulation or tetrahedralization $ \mathcal{T}_h$ is assumed to be quasi-uniform (see \cite{Br:2007}). In this space, the finite element approximation is: Find $u_h \in V_h$ such that

\begin{equation}
a(u_h,v) = (f,v) \quad \forall v \in V_h.
\end{equation}

\noindent This problem leads to the linear system 

\begin{equation}
\label{eq:linear_system}
A x = b. 
\end{equation}

\noindent Under our assumptions on $\sigma$ and $\kappa$, the matrix $A \in \mathbb{C}^{n\times n}$ will be non-normal, i.e.,

\begin{equation}
\label{eq:normal}
\mat{A} \mat{A}^* \neq \mat{A}^* \mat{A}.
\end{equation}

\noindent and indefinite.

In the following, we work with functions from the finite element space $V_h$ and the corresponding coefficient vectors. All functions from the finite element space can be expressed as $ u = \sum (\vec{x}_u)_i \varphi_i$, where $\varphi_i$ are the finite element basis functions. The vector of coefficients for the finite element function~$u$ will be denoted by $\vec{x}_u$. Using this notation, the system matrix $A$ is related to the sesquilinear form (\ref{eq:sesquilinear_form}) via $A_{j,i} = a(\varphi_i,\varphi_j)$, i.e., we have

\begin{equation*}
a(u,v) = \vec{x}_v^* A \vec{x}_u \quad \forall u,v \in V_h.
\end{equation*}

In the following, the norm equivalence between the Euclidian norm of vector $\vec{x}_u$ and the $L^2(\Omega)$-norm of a function $u$ is used.

\begin{lemma} Let $u \in V_h$ and $\Omega \subset \mathbb{R}^d$. Then there exists positive constants $c,C >0$, independent of $h$, such that
 
\label{le:norm_eq}
\begin{equation}
ch^d \left| \vec{x}_u \right|^2 \leq \| u\|_0^2 \leq Ch^d \left| \vec{x}_u \right|^2
\end{equation}
\end{lemma} 
\begin{proof}
See, e.g., \cite{Br:2007}
\end{proof}

\section{Convergence of GMRES}

The GMRES algorithm \cite{SaSc:86} approximately solves the linear system

\begin{equation}
A \vec{x} = \vec{b}
\end{equation}

\noindent by iteratively constructing the minimizer of the residual, $\left | A \vec{x} - \vec{b} \right|$, from the Krylov subspace $\mathcal{K}_m = \{\vec{b}, A \vec{b}, A^2 \vec{b},\ldots, A^{m-1} \vec{b} \}$. In exact arithmetic, the GMRES algorithm finds the exact solution in at most $n$~iterations for a general invertible matrix $A \in \mathbb{C}^{n \times n}$.

The convergence of the GMRES iteration can be improved by using right, left, or split preconditioners. The split preconditioner is used, when the preconditioner matrix can be decomposed into two parts, e.g. by using the Cholesky decomposition. As obtaining suitable decomposition for our preconditioner matrices is very costly,  using split preconditioner is not an option in our case.

Left preconditioning in GMRES leads to minimizing the residual 

\begin{equation*}
 \left | B A \vec{x} - B \vec{b} \right|.
\end{equation*}

\noindent In our case, the preconditioner matrix $B$ is ill-posed, so the iterative solution might be incorrect even for very small values of the residual. Due to this fact, we will only consider right preconditioning. In this case the residual to be minimized is

\begin{equation*}
\left | A B \tilde{\vec{x}} - \vec{b} \right|.
\end{equation*}

\noindent The actual solution is $\vec{x} = B \tilde{ \vec{x}}$, hence the right preconditioned GMRES minimizes the actual residual, regardless of the conditioning of the preconditioner matrix $B$. On a sufficiently fine grid, this quantity can be easily related to the error between the iterative solution and the exact solution to the linear system measured in the $H^1(\Omega)$-norm.

The convergence of the GMRES method is related to the minimization problem (see e.g. \cite{Gr:1997})

\begin{equation}
\label{eq:GMRES_minimization}
| \vec{r}_i | = \min_{p \in \tilde{P_i}} | p(A) \vec{r}_0 |, 
\end{equation}  

\noindent in which $\vec{r}_i$ is the residual on step $i$ and $\tilde{P}_i$ a monic polynomial of order $i$. As solving the minimization problem (\ref{eq:GMRES_minimization}) is far more costly than solving the original linear system, it is not a practical measure of the GMRES convergence rate. More useful bounds have been derived from (\ref{eq:GMRES_minimization}) in several alternative ways, depending on the properties of matrix $\mat{A}$ (see, e.g, \cite{Sa:2003,Gr:1997}). A good comparison of different GMRES convergence criterions is given in \cite{Em:1999}. 

If  the matrix $\mat{A}$~is normal, i.e., it satisfies equation (\ref{eq:normal}), the convergence is characterized only by the location of the eigenvalues of $A$. When the matrix $A$ is non-normal but diagonalizable, the eigenvectors of the matrix are not orthogonal and they have an effect on the convergence in addition to the location of the eigenvalues. The convergence of GMRES for general non-normal matrix equations can be also related to the properties of the pseudospectrum \cite{TrEm:2005} or the field of values \cite{Gr:1997}.  The FOV is defined as the set

\begin{equation}
\label{eq:FOV_set}
\mathcal{F}(A) = \left\{ \frac{ \vec{x}^* A \vec{x} }{ \vec{x}^* \vec{x} } \; \bigg | \;\vec{x} \in \mathbb{C}^n, \; \vec{x} \neq 0 \; \right\}.
\end{equation}

\noindent  Due to the connection between vectors of coefficients and finite element functions, the FOV is naturally related to the properties of the sesquilinear form $a(u,v)$. Hence,  we have chosen to use a FOV based convergence criterion in our analysis. 
 
The convergence of GMRES is related to the dimensions and the location of the set (\ref{eq:FOV_set}) in the complex plane.  Several different convergence estimates based on the FOV can be derived. A simple estimate is given in \cite{Gr:1997}, let $D = \left\{ \; z \in \mathbb{C} \; | \; \left| z - c \right| \leq s \; \right\}$  be a disc containing the FOV, but not the origin. Then, we have the convergence estimate

\begin{equation}
\label{eq:FOV_conv}
| \vec{r}_i | \leq \left( \frac{s}{| c | } \right)^i | \vec{r}_0 |.
\end{equation}

\noindent  In this work, we are mainly interested in studying the dependence of the convergence of preconditioned GMRES on the mesh size $h$ as well as parameters $\sigma$ and $\kappa$. As the bound (\ref{eq:FOV_conv}) remains unchanged under scaling of the coordinate system, it will be sufficient to study how the relative size of the FOV depends on these parameters.

The properties of the FOV have been extensively studied in the literature (see, e.g., \cite{HoJo:1991,Gr:1997,GuRa:97}). In the numerical results section, we will compute FOV for the preconditioned linear systems by using the procedure from \cite{Gr:1997}. This procedure is based on the the rotation property of the FOV, 

\begin{equation*}
\mathcal{F}( \mat{A} ) = \exp{-\mathrm{i} \theta} \mathcal{F} ( \exp{ \mathrm{i} \theta } \mat{A} )
\end{equation*}

\noindent and on the fact that  $\mathcal{F}( \mat{A} )$ is located on the left half plane from the largest eigenvalue of 

\begin{equation*}
H(\mat{A} ) = \frac{1}{2} \left( \mat{A} + \mat{A}^* \right).
\end{equation*}

\noindent By computing the largest eigenvalues for several rotated matrices $\exp{\mathrm{i} \theta} \mat{A}$, we will obtain a set containing $\mathcal{F}(\mat{A} )$.

\section{Laplace preconditioner for a  constant $\sigma$}

In this section, we consider using the finite element solution of the Poisson equation as a preconditioner for the Helmholtz equation. We will give bounds for the FOV of the preconditioned system, which gives a convergence estimate for GMRES via equation (\ref{eq:FOV_conv}). Our analysis is valid when the parameter $\sigma \in \mathbb{R}, \sigma > 0$. 

The preconditioner $P:V_h \rightarrow V_h$ is defined as: For each $u \in V_h$ find $Pu \in V_h$ such that
 
\begin{equation}
\label{eq:L_prec_oper}
( \nabla P u , \nabla v) = ( u ,v ) \quad \forall \quad v \in V_h.
\end{equation}

\noindent The matrix form of the operator $P$ is $ \mat{K}^{-1} \mat{M}$, where $\mat{K}$ is the stiffness matrix and $\mat{M}$ the mass matrix, i.e.,

\begin{equation}
\label{eq:K_and_M_matrix}
\vec{x}_v^* K \vec{x}_u = (\nabla u, \nabla v) \quad \mbox{and} \quad \vec{x}_v^* M \vec{x}_u = (u, v) \quad \forall u,v \in V_h.
\end{equation}

\noindent The right preconditioned linear system has the form 

\begin{equation}
\label{eq:L_prec}
\mat{A} \mat{K}^{-1} \mat{M} \tilde{\vec{x} } = \vec{b} .
\end{equation} 

\noindent We immediately observe, that

\begin{equation*}
\vec{x}_u^* \mat{A} \mat{K}^{-1} \mat{M} \vec{x}_u = a(Pu, u) \quad \forall u \in V_h.
 \end{equation*}

\noindent Using this connection, the FOV set of the preconditioned system can be written as

\begin{equation}
\label{eq:FOV_laplace}
\mathcal{F}(\mat{A} \mat{K}^{-1} \mat{M} ) = \left\{ \frac{ a(Pu,u) }{ \vec{x}_u^* \vec{x}_u} \; \bigg | \; \vec{x}_u \in \mathbb{C}^n, \; \vec{x}_u \neq 0 \; \right \}.
\end{equation}

\noindent To give bounds for this set, we will study the sesquilinear form $a(Pu,u)$ instead of working directly with the matrix $\mat{A} \mat{K}^{-1} \mat{M}$. A similar connection is also the basis for the derivation of the convergence estimates for additive Schwarz methods applied to elliptic problems. The main difference to our case is that we need to obtain estimates between the sesquilinear form and the Euclidian vector norm. For elliptic problems, similar estimates are derived in the $H^1(\Omega)$-norm.

By the definition of the sesquilinear form (\ref{eq:sesquilinear_form}), we have 

\begin{equation}
\label{eq:L_aPuu}
a ( P u, u) = (\nabla  P u, \nabla u) - \kappa^2 (Pu,u) + \mathrm{i} \sigma ( Pu,u).
\end{equation}

\noindent To bound the terms above, we use the following elementary result.

\begin{lemma}
\label{le:L_tools}
Let $P$ be defined as in (\ref{eq:L_prec_oper}) and $u \in V_h$. Then 

\begin{equation}
\label{eq:L_aPuu:1}
(Pu,u) = \| \nabla P u \|_0^2.
\end{equation}

\noindent and there exist a constant $C > 0$, independent of $h$,~$\sigma$,~and~$\kappa$, such that 

\begin{equation*}
\| \nabla P u \|_0 \leq C \| u \|_0.
\end{equation*}

\end{lemma}

\begin{proof} 
The equation (\ref{eq:L_aPuu:1}) follows directly from the definition of the operator $P$, equation (\ref{eq:L_prec_oper}),

\begin{equation*}
\| \nabla P u \|_0^2 = (u,P u) = ( P u, u).
\end{equation*}

\noindent Using the Cauchy-Schwartz inequality gives

\begin{equation*}
\| \nabla P u \|^2_0 \leq \| u \|_0 \| P u \|_0.
\end{equation*}

\noindent Applying the Poincare-Friedrichs inequality completes the proof.

\end{proof}

Using Lemma \ref{le:L_tools}, the definition of $P$, and equation (\ref{eq:L_aPuu}) gives 

\begin{equation}
\label{eq:L_aPuu2}
a(Pu,u) =  \| u \|_0^2 - \kappa^2 \| \nabla Pu \|_0^2 +  \mathrm{i} \sigma \| \nabla Pu \|_0^2.
\end{equation}

\noindent Based on this equation, it is straightforward to derive bounds for the FOV set. We begin with the obvious bounds.

\begin{theorem}
\label{th:L_rectangle}
There exists a constant $C>0$, independent of $h$, $\sigma$, and $\kappa$, such that

\begin{equation*}
\mathcal{F}(\mat{A} \mat{K}^{-1} \mat{M} )  \subset [C(1- \kappa^2) h^d,C h^d] \times [0, C \sigma h^d].
\end{equation*}

\noindent where $d$ is the spatial dimension.

\end{theorem}
\begin{proof}
By equation (\ref{eq:L_aPuu2}), we have 

\begin{equation}
\label{eq:L_th1_RE}
\Re \vec{x}_u^* \mat{A} \mat{K}^{-1} \mat{M} \vec{x}_u =  \| u \|_0^2 - \kappa^2 \| \nabla Pu \|_0^2 
\end{equation}

\noindent and

\begin{equation}
\label{eq:L_th1_IM}
\Im \vec{x}_u^* \mat{M} \mat{A}  \vec{x}_u =  \sigma \| \nabla Pu \|_0^2.
\end{equation}

\noindent In addition, we clearly have 

\begin{equation*}
\Re \vec{x}_u^* \mat{A} \mat{K}^{-1} \vec{x}_u  =  \| u \|_0^2 - \kappa^2 \| \nabla Pu \|_0^2 \leq \|u \|_0^2
\end{equation*}

\noindent and 

\begin{equation*}
\Im \vec{x}_u^* \mat{M} \mat{A} \vec{x}_u =  \sigma \| \nabla Pu \|_0^2 \geq 0.
\end{equation*}

\noindent Using Lemma \ref{le:L_tools} to estimate $\| \nabla P u\|_0$ in (\ref{eq:L_th1_RE}) and  (\ref{eq:L_th1_IM}), gives 

\begin{equation*}
\Re \vec{x}_u^* \mat{A} \mat{K}^{-1} \mat{M} \vec{x}_u \geq (1-C \kappa^2 ) \| u \|_0^2
\end{equation*}

\noindent and

\begin{equation*}
\Im \vec{x}_u^* \mat{M} \mat{A}  \vec{x}_u \leq C \sigma \| u \|_0^2.
\end{equation*}

\noindent The result follows from combining the above estimates with Lemma \ref{le:norm_eq}. 
\end{proof}

These bounds state that the FOV set is located inside a rectangle that contains the origin. In such a case, the convergence estimate (\ref{eq:FOV_conv}) does not deliver any information on the convergence of GMRES. Fortunately, we can improve the above bounds. 

\begin{theorem} 
\label{th:L_strip}
There exists a constant $C>0$, independent of $h$, $\sigma$ and $\kappa$, such that

\begin{equation}
\mathcal{F}(\mat{A} \mat{K}^{-1} \mat{M} )  \subset \left\{ z \in \mathbb{C} \; \bigg | \;  ch^d - \frac{\kappa^2}{\sigma} \Im z \leq \Re z \leq Ch^d - \frac{\kappa^2}{\sigma} \Im z   \; \right \}.
\end{equation}

\noindent where $d$ is the spatial dimension.

\end{theorem}
\begin{proof}

For each $z \in \mathcal{F}(AK^{-1}M )$ there exists a $\vec{x}_u \in \mathbb{C}^n$ and a corresponding $u \in V_h$ such that

\begin{equation}
\label{eq:L_th2_z}
z = \frac{ \vec{x}_u^* AK^{-1} M \vec{x}_u}{\vec{x}_u^*\vec{x}_u} = \frac{ a(Pu,u)}{\vec{x}_u^*\vec{x}_u}.
\end{equation}

\noindent  Taking the imaginary part of the above and using (\ref{eq:L_aPuu2}) yields 

\begin{equation}
\| \nabla P u\|_0^2 = \frac{ \Im z}{\sigma}  \vec{x}_u^*\vec{x}_u.
\end{equation}

\noindent Combining this equation with the real part of (\ref{eq:L_th2_z}) and using (\ref{eq:L_aPuu2}), we obtain

\begin{equation}
\Re z = \frac{\Re a(Pu,u)}{\vec{x}_u^*\vec{x}_u} = \frac{\| u \|_0^2}{\vec{x}_u^*\vec{x}_u} - \kappa^2 \frac{\Im z}{ \sigma}  
\end{equation}

\noindent Applying the Lemma \ref{le:norm_eq} completes the proof.

\end{proof}

Theorems \ref{th:L_rectangle} and \ref{th:L_strip} state that the FOV set for the preconditioned system is located at the intersection of a strip and a rectangle.  The intersection does not contain the origin, hence it can be used in connection with (\ref{eq:FOV_conv}) to give convergence estimates for the preconditioned GMRES method. As all the dimensions of the FOV have an $h^d$ - dependence, the relative size of the FOV does not change when the mesh is refined. Thus, GMRES for the preconditioned system will converge with the same rate independently of $h$. The convergence rate depends on the distance of the FOV set from origin and the size of the set. These parameters depend on $\frac{\kappa^2}{\sigma}$, $\kappa^2$ and $\sigma$, which will determine the convergence speed.

\section{Inexact Laplace Inverse}

In this section,  we consider replacing the exact solution of the linear system arising from discretization of the Poisson problem in the preconditioner (\ref{eq:L_prec}) by an approximate solution. The presented analysis is based on a perturbation argument and it is valid for inexactly solving (\ref{eq:L_prec_oper}) using a symmetric iterative method convergent in the $\| \cdot \|_0$ and $\| \cdot \|_1$ norms. Example of a method fitting to this category is the multigrid (MG) method (for simple convergence proofs, see e.g. \cite{Br:2007}). Our analysis indicates that the Poisson problem should be solved more accurately for large values of the parameter $\kappa$ for guaranteed convergence of the GMRES method.

In the following, we are solving the linear system

\begin{equation}
\label{eq:inexact_linear_system}
K\vec{x} = \vec{b}
\end{equation}

\noindent by using an iterative method. We shall denote one iteration cycle as $\tilde{K}^{-1}$ and $N$ cycles as $\tilde{K}^{-N}$. All iterations start from a zero initial guess, so that  the error after $N$ steps is $\vec{e}_N = \vec{x} -  \tilde{K}^{-N} \vec{b} $. Equation (\ref{eq:inexact_linear_system})  gives

\begin{equation*}
\vec{e}_N = ( \tilde{K}^{-N} - K^{-1} ) K \vec{x}.
\end{equation*} 

\noindent This motivates us to define an error propagation operator $E_N: V_h \rightarrow V_h$ such that

\begin{equation*}
e_{N} = E_N u.
\end{equation*}

\noindent The matrix form of this operator is

\begin{equation}
\label{eq:EN}
( \tilde{K}^{-N} - K^{-1} ) K.
\end{equation}

In the following, we assume that there exists constants $\gamma_0$ and $\gamma_1$, 

\begin{equation*}
0 \leq \gamma_i < 1, i = 0,1
\end{equation*}

\noindent as well as a constant $C>0$, independent on $\gamma_0$ and $\gamma_1$, such that 

\begin{equation}
\label{eq:ass_EN}
\| E_N u \|_1 \leq C \gamma^N_1 \| u \|_1 \quad \mbox{and} \quad \| E_N u \|_0 \leq C \gamma^N_0 \| u \|_0 \quad \forall u \in V_h.
\end{equation}

\noindent This simply means that the applied iteration converges in the $H^1(\Omega)$- and $L^2(\Omega)$-norms. Such an assumption is directly satisfied by several iterative methods, e.g., by the multigrid method.

We will denote the inexact preconditioner as  $\tilde{P}$. The matrix form of this operator is   

\begin{equation*}
\label{eq:ieL_prec}
\tilde{K}^{-N} M
\end{equation*}

\noindent The FOV set for the preconditioned system satisfies 

\begin{equation*}
\mathcal{F}( A \tilde{K}^{-N} M) \subseteq \mathcal{F}( A K^{-1} M) \oplus   \mathcal{F}( A ( \tilde{K}^{-N}-K^{-1} ) M).
\end{equation*}

\noindent  We estimate the size of this set by combining  the results from the previous section with a bound for the perturbation set 

\begin{equation}
\label{eq:FOV_perturbation}
\mathcal{F}( A ( \tilde{K}^{-N}-K^{-1} ) M).
\end{equation}

\noindent Using operator notation, bounding the perturbation set translates to giving bounds for

\begin{multline}
\label{eq:FOV_inexact}
a((\tilde{P} - P)u,u) = \\ (\nabla (\tilde{P} - P)u, \nabla u) - \kappa ( (\tilde{P} - P)u,u) + \mathrm{i}\sigma ((\tilde{P} - P)u,u).
\end{multline}

\noindent The last two terms in the above equation are estimated with the following lemma. 

\begin{lemma} There exists a positive constant $C > 0$, independent on $\kappa$,$\sigma$, $\gamma_0$ and $\gamma_1$, such that
\label{lemma:P1}
\begin{equation*}
| ( (\tilde{P} - P)u,u) | \leq C \gamma_1^N \|u \|^2_0 \quad \forall u \in V_h. 
\end{equation*}

\end{lemma}
\begin{proof}
Using the Cauchy-Schwartz and the Poincar\'e-Friedrichs inequalities gives

\begin{equation*}
| ( (\tilde{P} - P)u,u) | \leq C \| \nabla (\tilde{P} - P) u \|_0 \| u \|_0.
\end{equation*}
\noindent The semi-norm above can be evaluated as 

\begin{equation*}
\| \nabla (\tilde{P} - P) u \|_0  = \sup_{v \in V_h} \frac{ ( \nabla (\tilde{P} - P) u , \nabla v ) }{\| \nabla v \|_0}.
\end{equation*}

\noindent Using the matrix form of the error propagation operator (\ref{eq:EN}) and symmetry as well as definitions (\ref{eq:L_prec})  and (\ref{eq:ieL_prec}) gives 

\begin{equation}
( \nabla (\tilde{P} - P) u , \nabla v ) = \vec{x}^*_v K ( K^{-1} - \tilde{K}^{-N}) M \vec{x}_u = ( u, E_N v). 
\end{equation}

\noindent Hence, 

\begin{equation*}
\| \nabla (\tilde{P} - P) u \|_0  = \sup_{v \in V_h} \frac{ ( u , E_N v ) }{\| \nabla v \|_0}.
\end{equation*}

\noindent By the Poincar\'e-Friedrichs inequality and the assumption (\ref{eq:ass_EN}) 

\begin{equation*}
 \| \nabla (\tilde{P} - P) u \|_0   = \sup_{v \in V_h} \frac{ (u, E v) }{ \| \nabla v \|_0 } \leq C \gamma_1^N \| u \|_0,
\end{equation*}

\noindent which completes the proof.

\end{proof}

Estimating the first term in the right hand side of (\ref{eq:FOV_inexact}) is straightforward. The matrix form of the error propagation operator (\ref{eq:EN}) and Cauchy-Schwartz inequality yield

\begin{equation}
\label{eq:FOV_inexact_estimate1}
(\nabla (\tilde{P} - P)u, \nabla u)  = (u,Eu) \leq \| Eu \|_0 \| u \|_0 \leq \gamma_0^N \|u\|_0^2.
\end{equation}

\noindent Combining this equation with Lemma \ref{lemma:P1} yields the following theorem.

\begin{theorem} 
\label{th:pert_bounds}
There exists a constant $C>0$, independent of $\gamma_0$, $\gamma_1$, $\kappa$, $h$, and $\sigma$, such that 

\begin{equation*}
\Re \mathcal{F}( A ( \tilde{K}^{-N}-K^{-1} ) M) \subset \left[ -C h^d ( \gamma_0^N + \kappa^2 \gamma_1^N),  C h^d(\gamma_0^N + 
 \kappa^2 \gamma_1^N) \right]
\end{equation*}

\noindent \mbox{and}

\begin{equation*}
\Im \mathcal{F}( A ( \tilde{K}^{-N}-K^{-1} ) M) \subset \left[ -Ch^d (\gamma_0^N +\sigma \gamma_1^N), Ch^d(\gamma_0^N +  \sigma \gamma_1^N) \right].
\end{equation*}

\noindent where $d$ is the spatial dimension and $N$ the number of iterations.

\end{theorem}
\begin{proof} We will use the connection between the FOV set and the sesquilinear form to derive the bound. First, we observe that

\begin{equation*}
( ( \tilde P - P) u,u) = \vec{x}^*_u M (\tilde{K}^{-N} - K^{-1} ) M \vec{x}_u
\end{equation*}

\noindent as the applied iterative method is assumed symmetric, the term 

\begin{equation*}
\vec{x}^*_u M (\tilde{K}^{-N} - K^{-1} ) M \vec{x}_u 
\end{equation*}

\noindent is real valued. Hence,

\begin{equation*}
\Re \mathcal{F}( A ( \tilde{K}^{-N}-K^{-1} ) M) = \Re  (\nabla (\tilde{P} - P)u, \nabla u) - \kappa^2 ( (\tilde{P} - P)u,u)
\end{equation*}

\noindent and

\begin{equation*}
\Im \mathcal{F}( A ( \tilde{K}^{-N}-K^{-1} ) M) = \Im  (\nabla (\tilde{P} - P)u, \nabla u) + \sigma ( (\tilde{P} - P)u,u).
\end{equation*}

\noindent Combining these equations with the equation (\ref{eq:FOV_inexact_estimate1}), Lemma \ref{lemma:P1} and Lemma \ref{le:norm_eq}, completes the proof.

\end{proof}

Theorem \ref{th:pert_bounds} states that the perturbation set is located inside a rectangle. The dimensions of this rectangle are dependent on the number of iterations taken with the iterative scheme and on the parameters $\sigma$ and $\kappa$. The estimate states, how the size of the perturbation set $\mathcal{F}( A ( \tilde{K}^{-N}-K^{-1} ) M)$ converges to zero when the number of iterations $N$ is increased.

From theoretical point of view, the implication of Theorem \ref{th:pert_bounds} is that the number of iterations should be increased when the parameter $\kappa$ grows to keep the size of the perturbation set small and the origin outside the actual FOV. In section 7, we will numerically demonstrate that the presented analysis captures the behavior of the perturbation set. The number of GMRES iterations required to solve the preconditioned problem is strongly affected by the number of iterations $N$. However, the inclusion of the origin into the FOV seems to have a very small effect on the convergence of GMRES. This reflects the fact that the convergence estimate (\ref{eq:FOV_conv}) is only an upper bound for the required number of iterations.

\section{Two-level preconditioner for a general $\sigma$ }

\noindent In this section, we consider a two-level preconditioner combining a solution to the Poisson problem with a coarse grid correction. Our analysis allows the parameter $\sigma$ to be a function satisfying the assumption given in equation (\ref{eq:general_sigma}). 

The idea behind the two-level preconditioners for the Helmholtz equation is to use a coarse grid solution to eliminate all eigenfrequencies corresponding to an eigenvalue with a negative real part. Then a Poisson problem is solved to approximate the solution components corresponding to eigenvalues with a positive real part. Such methods are successful in keeping the number of iterations constant, but impose same mesh size constraints as for the computationl grid also on the coarse grid.

In previous works on this type of preconditioners \cite{BrPaLe:1993,CaWi:1993}, the $\kappa$-dependency of the coarse grid mesh size has not been explicitly studied. A widely acknowledged rule of thumb is that the coarse grid has to satisfy the same mesh size requirements as the original grid. We will verify this rule by giving bounds for the coarse grid mesh size $H$ depending on $\kappa$, $\sigma$ and the stability constant $C_S$. By Theorem \ref{th:stability}, the stability constant for $\sigma \in \mathbb{R}$ is $C_S =  \kappa^2 \sigma^{-1}$. As we will see, in the worst case this estimate leads to the requirement

\begin{equation*}
\kappa^3 H \ll 1 \quad
\end{equation*}

\noindent for the coarse grid mesh size $H$. Based on our analysis, it is easy to see that the same mesh size constraint should also be valid for the actual computational grid. One should note that the bound $\kappa^2 h \ll 1$ given in \cite{BaIh:95,BaIh:97} is obtained for a problem with different boundary conditions and hence a different stability estimate. In addition, our analytical results consider only the so-called asymptotic range, whereas the more elaborate analysis in \cite{BaIh:95,BaIh:97} take also into account the pre-asymptotic range, i.e. the case when the mesh size requirement is not satisfied.

Our preconditioner has the matrix form 

\begin{equation}
\label{eq:cg_prec}
R_H A_H^{-1} R_H^T M + K^{-1} ( I - A R_H A_H^{-1} R_H^T ) M,
\end{equation}

\noindent  where $A_H$ is the system matrix from the space $V_H$ and $K$ as well as $M$ are as defined in (\ref{eq:K_and_M_matrix}). The matrix $R_H$ is the prolongation operator from $V_H$ to $V_h$. The mass matrix is included to the preconditioner (\ref{eq:cg_prec}) to aid in formulating it as an operator in the following analysis. The term $R_H A_H^{-1} R_H^T M$ corresponds to solving the original problem in the space $V_H$ and the term $K^{-1} ( I - A R_H A_H^{-1} R_H^T) M$ to solving the Poisson problem for the residual.

As in the previous sections, we will derive an estimate for the FOV of the preconditioned linear system. In order to do this we will first interpret the preconditioner (\ref{eq:cg_prec}) as a sum of two operators, $P_H$ and $Q$. The operator $P_H$ corresponds to the solution of the original problem in the coarse space and $Q$ to the solution of the Poisson problem for the residual.

The operator $P_H: V_h \rightarrow V_H$ is defined as: For each $u\in V_h$ find $P_H u \in V_H$ such that

\begin{equation}
\label{eq:PHoper}
a( P_H u , v_H) = (u,v_H) \quad \forall v_H \in V_H. 
\end{equation}

The matrix form of this operator is $R_H A_H^{-1} R_H^T M$. To define the operator $Q$, we first introduce an $a(u,v)$ - orthogonal projection operator $I_H:V_h \rightarrow V_H$. This operator is defined as: For each $u\in V_h$ find $I_Hu \in V_H$ such that  

\begin{equation}
\label{eq:IHoper}
a( v, I_H u ) = a(v,u) \quad \forall v \in V_H.
\end{equation}

\noindent The matrix form of this operator is $R_H A_H^{-*} R_H^T A^{*}$. An immediate consequence of this definition is the orthogonality property

\begin{equation}
\label{eq:IHortho}
a(v_H,(I-I_H) u) = 0 \quad \forall v \in V_H.
\end{equation}

\noindent This property will be used frequently in the following analysis. 

The term $K^{-1} ( I - A R_H A_H^{-1} R_H^T ) M$  in the preconditioner (\ref{eq:cg_prec}) corresponds to the operator $Q:V_h \rightarrow V_h$ defined as: For each $u\in V_h$ find $Q \in V_h$ such that 

\begin{equation}
\label{eq:Qoper}
(\nabla Q u,\nabla v) = ( u, (I - I_H) v) \quad \forall v \in V_h.
\end{equation}

The main task in analyzing the two-level preconditioned is to derive a convergence result for the operator $Q$. This result is obtained by studying the properties of the projection operator $I_H$. The applied techniques are rather standard for the analysis of the finite element discretizations of the Helmholtz equation. However,  as we want to explicitly state the dependency of the required coarse grid mesh size on $H$, $\kappa$, $\sigma$, and $C_S$  some modifications to the approach used in previous works, \cite{Sc:74,BrPaLe:1993,CaWi:1993} has been made. In the derivation of the convergence estimate, we will use the standard nodal interpolation operator to the coarse space, $\pi_H$. For a function $u\in H^{3/2+\delta}(\Omega)$, $\delta \in (0,1/2]$ this operator has the approximation property: 

\begin{equation}
\label{eq:interpolation1}
\| u - \pi_H u \|_1 \leq C H^{ 1/2+\delta} |u|_{3/2+\delta} \quad \mbox{and} \quad \| u - \pi_H u \|_0 \leq C H^{3/2+\delta} |u|_{3/2+\delta},
\end{equation}

\noindent where the constant $C>0$ is independent of $H$. For a function $u\in H^1(\Omega)$, the convergence is obtained only in the $L^2(\Omega)$-norm

\begin{equation}
\label{eq:interpolation2}
\| u - \pi_H u \|_0 \leq CH \|\nabla u \|_0,
\end{equation}

\noindent where the constant $C>0$ is independent of the mesh size $H$. Proofs of these approximation results can be found, e.g., from \cite{Br:2007}. We begin with a standard convergence estimate for $I_H$.

\begin{lemma} There exists a constant $C>0$, independent of $h$, $H$, $C_S$, $\kappa$ and $\sigma$,0 such that
 
\label{th:cg_approximation} 
\begin{equation}
\| (I-I_H) u \|_0 \leq \frac{C C_S H^{1/2+\delta}}{ 1 - C C_S (\kappa^2 + \|\sigma\|_\infty )H^{3/2+\delta}   } \|\nabla (I-I_H) u \|_0
\end{equation}

\noindent for some $\delta \in (0,1/2]$ and $H$ sufficiently small.

\end{lemma}

\begin{proof}
The proof is based on the standard duality argument (see e.g. \cite{Br:2007}). The $L^2$-norm can be evaluated as

\begin{equation}
\label{eq:duality_L2}
\| (I - I_H) u \|_0 = \sup_{v \in V_h} \frac{ | ( (I-I_H) u, v)_0| }{\| v \|_0}.
\end{equation}

\noindent The dual problem is: Find $\varphi \in H^1_0(\Omega)$ such that

\begin{equation}
\label{eq:dual_problem}
a(\varphi, w) = (w,v) \quad \forall w \in H^1_0(\Omega).
\end{equation}

\noindent Choosing the test function as $w = (I-I_H)u$ and using the orthogonality property (\ref{eq:IHortho}) gives

\begin{equation*}
(v, (I-I_H)u) = a(\varphi-\pi_H \varphi,(I-I_H)u), 
\end{equation*}

\noindent where $\pi_H \varphi \in V_H $ is the nodal interpolant of $\varphi$. The interpolation error estimates given in equation (\ref{eq:interpolation1}) and regularity Theorem \ref{th:regularity} gives

\begin{multline*}
| a(  \varphi - \pi \varphi, (I-I_H)u) | \leq  C C_S H^{1/2+\delta} \|\nabla (I-I_H)u \|_0 \| v \|_0  \\ + C C_S (\kappa^2 + \| \sigma \|_\infty) H^{3/2+\delta }\| (I-I_H)u \|_0  \| v \|_0.
\end{multline*}

\noindent Combining this estimate with (\ref{eq:duality_L2}) and reorganizing the terms yields

\begin{equation*}
\| (I-I_H) u \|_0 \leq \frac{C C_S H^{1/2+\delta}}{ 1 - C C_S (\kappa^2+ \|\sigma\|_\infty)H^{3/2+\delta}   } \| \nabla (I-I_H) u \|_0,
\end{equation*}

\noindent which is valid for $H$ such that $1 - C C_S (\kappa^2 + \|\sigma\|_\infty)H^{3/2+\delta}  > 0$.
\end{proof}

The above Theorem gives a convergence result for the operator $I_H$ if the coarse grid satisfies the condition  

\begin{equation*}
1 - C C_S (\kappa^2 + \| \sigma \|_\infty )H^{3/2+\delta}  > 0
\end{equation*}

\noindent The parameter $\delta$ in the estimate depends on the regularity of the solution to the dual problem (\ref{eq:dual_problem}). As the load for the dual problem is always from the space $H^1(\Omega)$, the regularity depends completely on the shape of the domain. By Theorem \ref{th:regularity}, this dependency is the same as for the Poisson problem. In the following, we will state the coarse gird mesh size requirements under the assumption $\| \sigma \|_\infty \ll \kappa$. Under this assumption, the $H^2(\Omega)$ regularity of the dual problem, and in the light of the stability estimate given in Theorem \ref{th:stability}, the mesh size requirement given in Lemma \ref{th:cg_approximation} translates to 

\begin{equation*}
\kappa^2 H \ll 1.
\end{equation*}

In the following, we will use Lemma \ref{th:cg_approximation} in a shorter form: There exists a positive constant $C > 0 $ such that

\begin{equation*}
\| (I-I_H) u \|_0 \leq C C_S H^{1/2+\delta} \| \nabla (I-I_H) u \|_0.
\end{equation*}

\noindent The first implication of Lemma \ref{th:cg_approximation} is that the sesquilinear form 

\begin{equation*}
a( (I-I_H)u,(I-I_H)u )
\end{equation*}

\noindent behaves in some sence as an coersive operator. This results is required in the following analysis to obtain a $\kappa$-independent bound for the term $\| \nabla (I- I_H)u \|_0$.

\begin{lemma} 
\label{le:cg_elliptic}
Let $u \in H_0^1(\Omega)$ and let the coarse grid mesh size $H$ be such that

\begin{equation}
\label{eq:H_ass}
\kappa^2 C_S H^{3/2+\delta} \ll 1 \quad \mbox{and} \quad \kappa^2 C^2_S H^{1+2\delta} \ll 1.
\end{equation}

\noindent Then there exist a constant $\alpha > 0$, independent on $h$,$H$,$C_S$,$\kappa$, and $\sigma$, such that

\begin{equation*}
\Re a((I-I_H)u,(I-I_H)u) \geq \alpha \| \nabla(I-I_H) u \|_0^2.
\end{equation*}

\end{lemma}
\begin{proof} By the definition of the sesquilinear form, we have

\begin{equation*}
\Re a( (I-I_H)u,(I-I_H) u) \\ = \| \nabla (I-I_H) u \|_0^2 - \kappa^2 \| (I-I_H) u \|_0^2.
\end{equation*}

\noindent Due to the assumptions (\ref{eq:H_ass}), we can apply Lemma \ref{th:cg_approximation}. This yields

\begin{equation*}
\Re a( (I-I_H)u,(I-I_H) u) \geq (1 - C \kappa^2 C_S^2 H^{1+2\delta} ) \| \nabla (I-I_H) \|^2_0.
\end{equation*}

\noindent Hence, the coarse grid mesh size has to satisfy the requirement 

\begin{equation*}
1 - C \kappa^2 C_S^2 H^{1+2\delta} > 0.
\end{equation*}

\noindent This is, $\kappa^2 C_S^2 H^{1+2\delta} \ll 1$.
\end{proof}

The major implication of the above result is the requirement 

\begin{equation*}
\kappa^2 C_S^2 H^{1+2\delta} \ll 1
\end{equation*}

\noindent For $\sigma \in \mathbb{R}, \sigma \ll \kappa^2$ and the $H^2(\Omega)$ regularity of the dual problem (\ref{eq:dual_problem}), this implies the requirement $\kappa^6 \sigma^{-2} H^2 \ll 1$. Hence, in the worst case the coarse grid has to satisfy the requirement 

\begin{equation*}
\kappa^3 H \ll 1.
\end{equation*}

\noindent A second important consequence of  Lemma \ref{le:cg_elliptic} is the $\kappa$-independent boundedness of the operator $I_H$. This result is required, when deriving convergence estimates for the operator $Q$.

\begin{lemma} \label{le:cg_bounded}
Let the coarse grid mesh size $H$ be such that the assumptions (\ref{eq:H_ass}) are satisfied and let $u\in H^1_0(\Omega)$. Then there exist a constant $C > 0$, independent on $h$,$H$,$C_S$, $\kappa$, and $\sigma$, such that

\begin{equation*}
\| \nabla (I-I_H) u \|_0 \leq C \| \nabla u \|_0.
\end{equation*}

\end{lemma}
\begin{proof} Under the assumptions (\ref{eq:H_ass}), Lemma \ref{le:cg_elliptic} states that there exists a constant $\alpha>0$, independent of $h$, $H$, $C_S$, $\kappa$, and $\sigma$, such that

\begin{equation*}
\alpha \| \nabla (I-I_H) u \|^2_0 \leq \Re a( (I-I_H)u,(I-I_H) u).
\end{equation*}

\noindent Now, let $\pi_H u \in V_H$  be the nodal interpolant of $u$. By the orthogonality property given in equation (\ref{eq:IHortho}), we have

\begin{equation*}
\alpha \| \nabla (I-I_H) u \|^2_0 \leq \Re a( (I-\pi_H)u,(I-I_H) u).
\end{equation*}

\noindent That is

\begin{multline*}
\alpha \| \nabla (I-I_H) u \|^2_0 \leq \Re \left( \nabla(I-\pi_H)u,\nabla (I-I_H) u) \right. \\ \left. - \kappa^2 ( (I-\pi_H)u, (I-I_H) u) + \mathrm{i} (\sigma (I-\pi_H)u, (I-I_H) u) \right) .
\end{multline*}

\noindent Using the Cauchy-Schwartz inequality, the interpolation error estimate (\ref{eq:interpolation2}), and the boundedness of the interpolation operator, we get

\begin{multline*}
\alpha \| \nabla (I-I_H) u \|^2_0 \leq \| \nabla u \|_0 \| \nabla (I-I_H)u \|_0 + C (\|\sigma\|_\infty + \kappa^2) H \| \nabla u\|_0 \|(I-I_H) u \|_0.
\end{multline*}

\noindent Applying the Poincar\'e-Friedrichs inequality and dividing by $\|\nabla (I-I_H)u\|_0$ completes the proof.
\end{proof}

A direct consequence of Lemmas \ref{th:cg_approximation} and \ref{le:cg_bounded} is an approximation property for the operator $Q$. The approximation property is obtained both in the $H^1(\Omega)$- and $L^2(\Omega)$-norms. The convergence result in the $L^2(\Omega)$-norm follows from the $H^1(\Omega)$ approximation property and the duality argument.

\begin{lemma}
\label{lemma:cg_Q1}
Let the coarse grid mesh size $H$ be such that the assumptions given in equation (\ref{eq:H_ass}) are satisfied. In addition let $u \in V_h$. Then there exist a constant $C > 0$, independent of $h$,$H$,$C_S$, $\kappa$, and $\sigma$, such that

\begin{equation}
\| \nabla Q u \|_0 \leq C C_S H^{1/2+\delta} \| u \|_0 \quad \forall u \in V_h.
\end{equation}

\noindent for some $\delta \in (0,1/2]$.
\end{lemma}

\begin{proof}
By the definition of operator $Q$, equation (\ref{eq:Qoper}), we have 

\begin{equation*}
\| \nabla Q u \|^2_0  = (u, (I-I_H) Qu).
\end{equation*}

\noindent The Cauchy-Schwartz inequality and Lemma \ref{th:cg_approximation} yield

\begin{equation*}
| (u, (I-I_H) Qu) | \leq C C_S H^{1/2+\delta} \| u \|_0  \| \nabla (I-I_H) Qu \|_0, 
\end{equation*}

\noindent Applying the boundedness result from Lemma \ref{le:cg_bounded} completes the proof. 
\end{proof}

\begin{lemma} 
\label{le:cg_Q0}
Let the coarse grid mesh size $H$ be such that assumptions given in equation (\ref{eq:H_ass}) are satisfied. In addition let $u \in V_h$. Then there exist a constant $C > 0$, independent of $h$,$H$,$C_S$, $\kappa$, and $\sigma$, such that

\begin{equation}
\| Q u \|_0 \leq C C_S H^{1+2\delta} \| u \|_0 \quad \forall u \in V_h.
\end{equation}

\noindent for some $\delta \in (0,1/2]$. 

\end{lemma}

\begin{proof}

\noindent Let $\varphi \in H^1_0(\Omega)$ be such that 

\begin{equation*}
(\nabla \varphi,\nabla v) = (Qu,v) \quad \forall v \in H^1_0(\Omega).
\end{equation*}

\noindent As $Qu \in V_h$, also $Qu \in L^2(\Omega)$. Hence, by elliptic regularity theory, $\varphi \in H^{3/2+\delta}(\Omega)$ and $\| \varphi \|_{3/2+\delta} \leq C \| Qu\|_0$. By the definition  of $Q$, we have the orthogonality property

\begin{equation*}
(\nabla \pi_H \varphi,\nabla Qu) = 0.
\end{equation*}

\noindent By this property and definition of opeartor $Q$, equation (\ref{eq:Qoper}), we have

\begin{equation*}
(\nabla (\varphi - \pi_H \varphi),\nabla Qu) = \| Q u \|_0^2.
\end{equation*}

\noindent Next, the Cauchy-Schwartz inequality leads to 

\begin{equation*}
\| Q u \|_0^2 \leq \| \nabla (\varphi - \pi_H \varphi )\|_0 \| \nabla Qu \|_0 
\end{equation*}

\noindent Applying the interpolation error estimate given in equation (\ref{eq:interpolation1}), the stability estimate for $\varphi$, and Lemma \ref{lemma:cg_Q1} gives

\begin{equation*}
\| Q u \|_0^2 \leq \ C C_S H^{1+2 \delta} \| Qu \|_0 \| u\|_0.
\end{equation*}

\noindent Dividing with $\| Qu \|_0$ completes the proof.

\end{proof}

Now, we are finally at the position to give bounds for the FOV. The preconditioner (\ref{eq:cg_prec}) has the operator form $P_H + Q$, so we need to bound

\begin{equation*}
a( (P_H + Q)u,u).
\end{equation*}

\noindent From the definition of the interpolation operator $I_H$ given in equation (\ref{eq:IHoper}) and the definition of operator $P_H$ given in equation (\ref{eq:PHoper}), it immediately follows that

\begin{equation*}
a(P_H u, u) = a( P_H u, I_H u) = (u,I_H u).
\end{equation*}

\noindent The definition of the operator $Q$ given in equation (\ref{eq:Qoper}), yields

\begin{equation*}
a(Q u,u) = ( u, (I - I_H) u)_0 - \kappa^2( Qu , u)_0 + \mathrm{i}(\sigma Qu,u).
\end{equation*}

\noindent Combining the two above identities gives

\begin{equation}
\label{eq:2L_prec_system}
a( (P_H + Q)u,u)  = \| u \|_0^2 - \kappa^2 (Qu,u)_0 + \mathrm{i}( \sigma Qu , u).
\end{equation}

\noindent Estimating the last two terms above using Lemma \ref{le:cg_Q0} leads to the desired bounds for the FOV.

\begin{theorem} Let the coarse grid mesh size $H$ satisfy the assumptions
\label{th:2L_bounds}
\begin{equation}
\label{eq:2L_ass}
k^2 C^2_S H^{1+2\delta} \ll 1 \quad \mbox{and} \quad k^2 C_S H^{3/2+\delta} \ll 1.
\end{equation}

\noindent Then there exists constants $c,C > 0$, independent of $h$, $H$, $\kappa$, and $\sigma$, such that

\begin{equation*}
\Re \mathcal{F} \subset  \left[ ch^d \left( 1 - C_S (\kappa^2 + \| \sigma \|_\infty) H^{1+2\delta} \right) ,Ch^d \left(1 + C_S (\kappa^2 + \| \sigma \|_\infty) H^{1+2\delta} \right) \right]
\end{equation*}

\noindent and 

\begin{equation*}
\Im \mathcal{F} \subset  \left[ ch^d C_S (\kappa^2 + \| \sigma \|_\infty) H^{1+2\delta}, Ch^d C_S (\kappa^2 + \| \sigma \|_\infty) H^{1+2\delta} \right]
\end{equation*}

\noindent where $d$ is the spatial dimension and $\delta \in (0,1/2]$.

\end{theorem}
\begin{proof} By equation (\ref{eq:2L_prec_system}), we have

\begin{equation}
\label{eq:2Lfov1}
\Re a( (P_H+Q)u, u) = \| u \|_0^2 + \Re \left( \mathrm{i}  (\sigma Qu,u) -\kappa^2  (Qu,u)\right) 
\end{equation}

\noindent and

\begin{equation}
\label{eq:2Lfov2}
\Im a( (P_H+Q)u, u) = \Im \left(\mathrm{i} (\sigma Qu,u) - \kappa^2 (Qu,u) \right).
\end{equation}

\noindent Under the assumptions made in equation (\ref{eq:2L_ass}), we can use the convergence estimates for term $\|Qu\|_0$ given in Lemma \ref{le:cg_Q0} to bound the above terms. This leads to the estimates

\begin{equation*}
|(Qu,u)| \leq  C C_S H^{1+2\delta} \| u \|_0^2 
\end{equation*}

\noindent and

\begin{equation*}
|(\sigma Qu,u)| \leq \| \sigma \|_\infty C C_S H^{1+2\delta} \| u \|_0^2.
\end{equation*}

\noindent Combining these inequalities and Lemma \ref{le:norm_eq} with equations (\ref{eq:2Lfov1}), (\ref{eq:2Lfov2}), completes the proof.

\end{proof}

The bounds given in the above theorem reflect the requirement on the approximation properties of the space $V_H$. The coarse grid has to be sufficiently dense, before the FOV set is completely located at the right-half plane. Based on the above result,  this happens in a convex domain when $\kappa^2 H \ll 1$. However, to obtain the estimates applied to derive the result the coarse grid mesh size requirement $\kappa^3 H \ll 1$ was made. Hence, the coarse grid mesh should satisfy the very strict requirement $\kappa^3 H \ll 1$.

\section{Numerical Examples}
\label{sec:numerical} 

In this section, we study the presented theory by numerical examples. We verify the bounds derived for the FOV by computing the actual sets with the procedure presented in Section 3. We will also solve the preconditioned linear system to show  the relationship between the FOV and the number of GMRES iterations required to solve the problem.

\begin{figure}[ht]
\begin{center}
\includegraphics[scale=0.5]{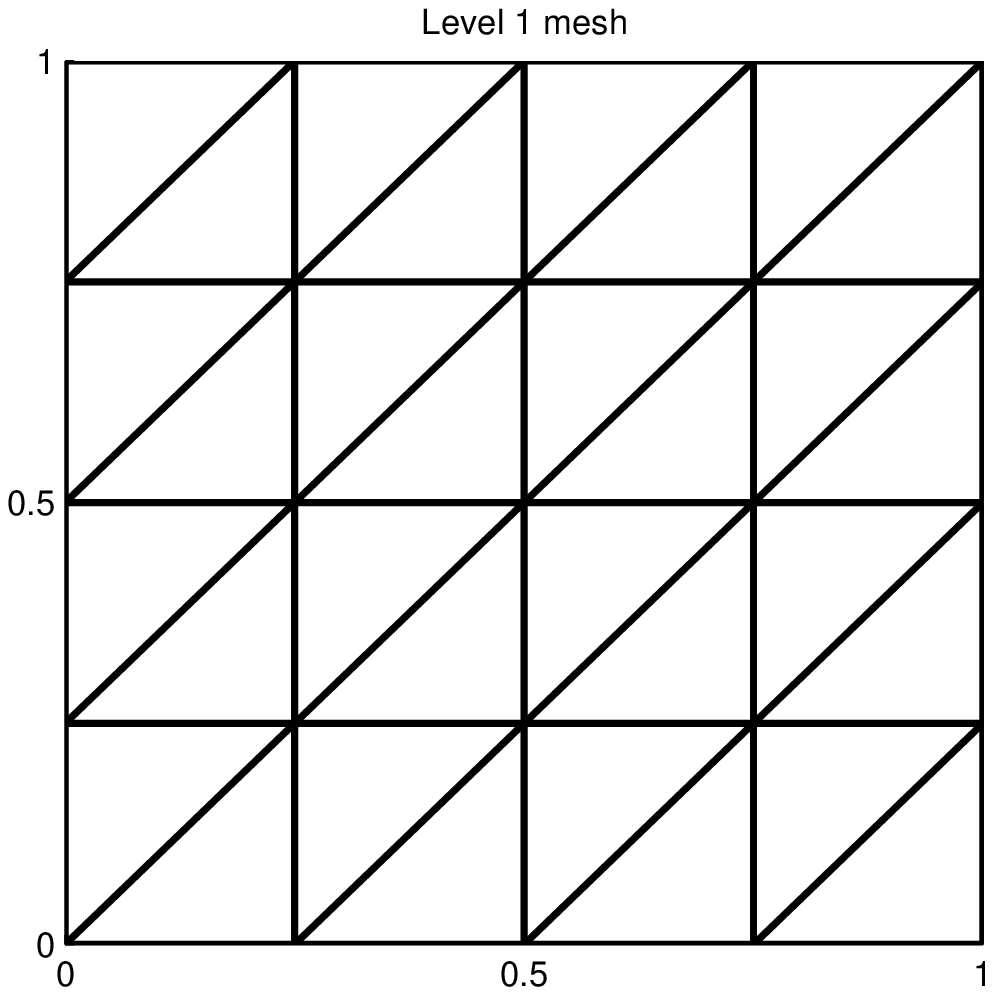} \hfill
\includegraphics[scale=0.5]{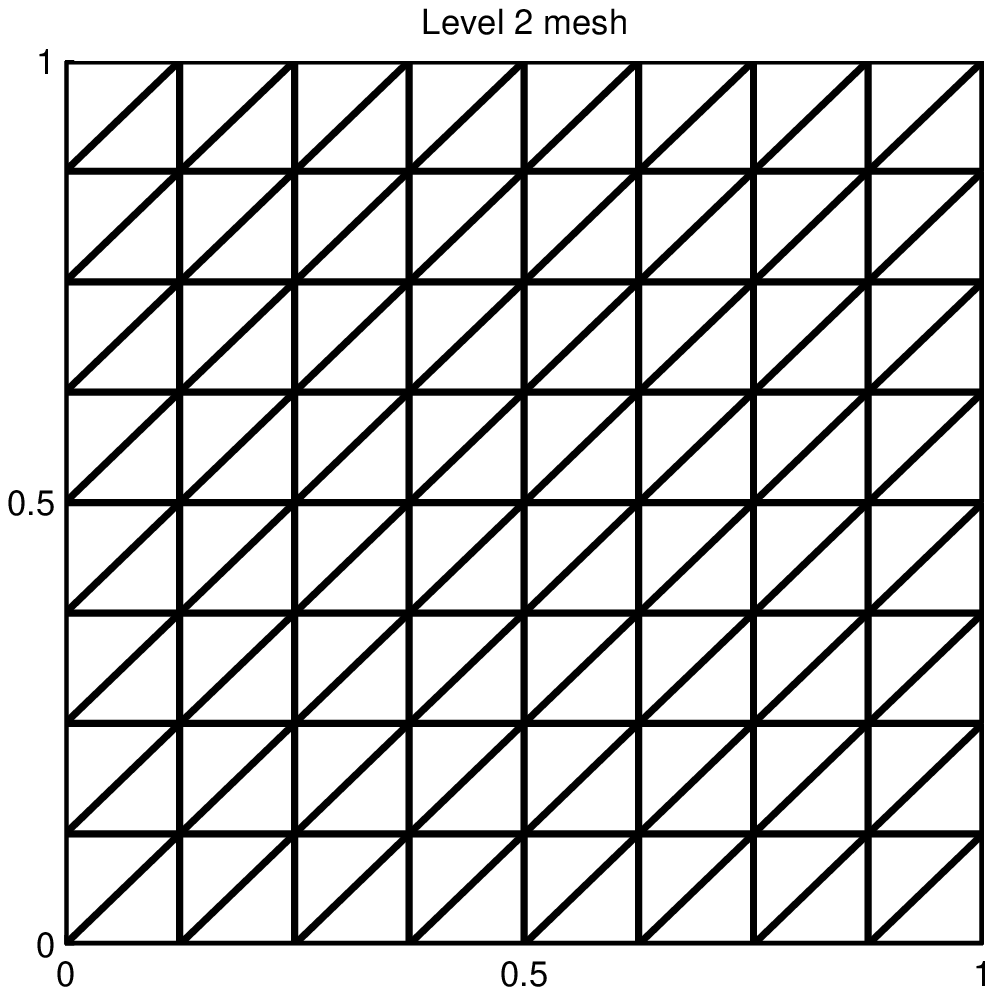}
\end{center}
\caption{The first two mesh levels for the two dimensional test case.}
\label{fig:test1_mesh_levels}
\end{figure}

The numerical tests are composed of two parts. First, we will consider the two dimensional unit square, $\Omega=(0,1)^2$. In this test, we use a family of triangulations composed by uniformly refining an initial mesh. The first two mesh levels are presented in Fig. \ref{fig:test1_mesh_levels}. 

The second part is a computationally more realistic three dimensional test case. As the number of degrees of freedom required in a three dimensional domain is considerably larger compared to the two dimensional case, we have not computed the FOV. The  main focus is in the convergence of the  preconditionerd GMRES method.

\subsection{Laplace preconditioner}

We begin by studying the exact Laplace preconditioner presented in Section 4. Our aim is to verify the bounds given for the FOV in Theorems \ref{th:L_rectangle} and \ref{th:L_strip}. These theorems state that the FOV for the Laplace preconditioned system in two dimensions is contained inside a rectangle $(-c\kappa^2 h^2,Ch^2) \times (0,C\sigma h^2)$ and the strip 

\begin{equation}
\label{eq:num:strip}
 \left\{ z \in \mathbb{C} \; \bigg | \;  ch^2 - \frac{\kappa^2}{\sigma} \Im z \leq \Re z \leq Ch^2 - \frac{\kappa^2}{\sigma} \Im z   \; \right \},
\end{equation}

\noindent in which constants $c,C>0$ are independent of $\kappa,\sigma$ and $h$. 

\begin{figure}[ht]
\begin{center}
\includegraphics[scale=0.5]{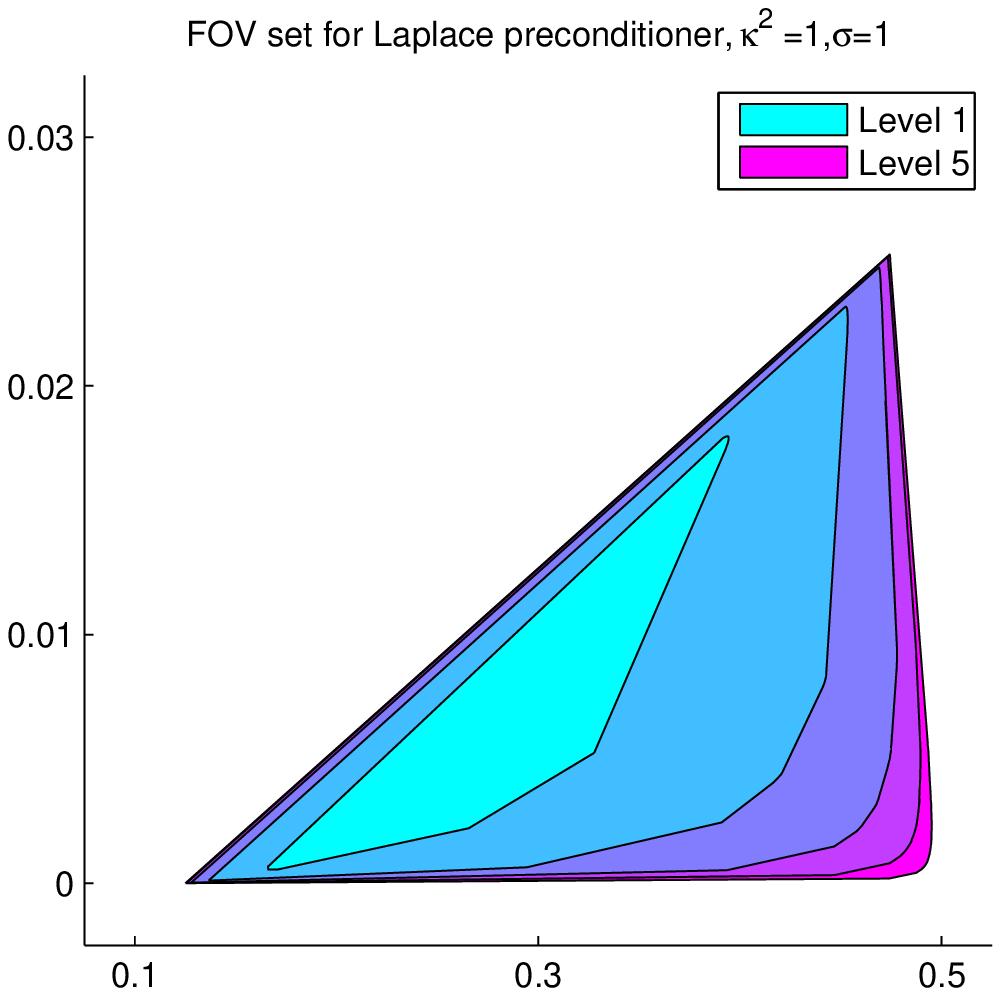} \hfill
\includegraphics[scale=0.5]{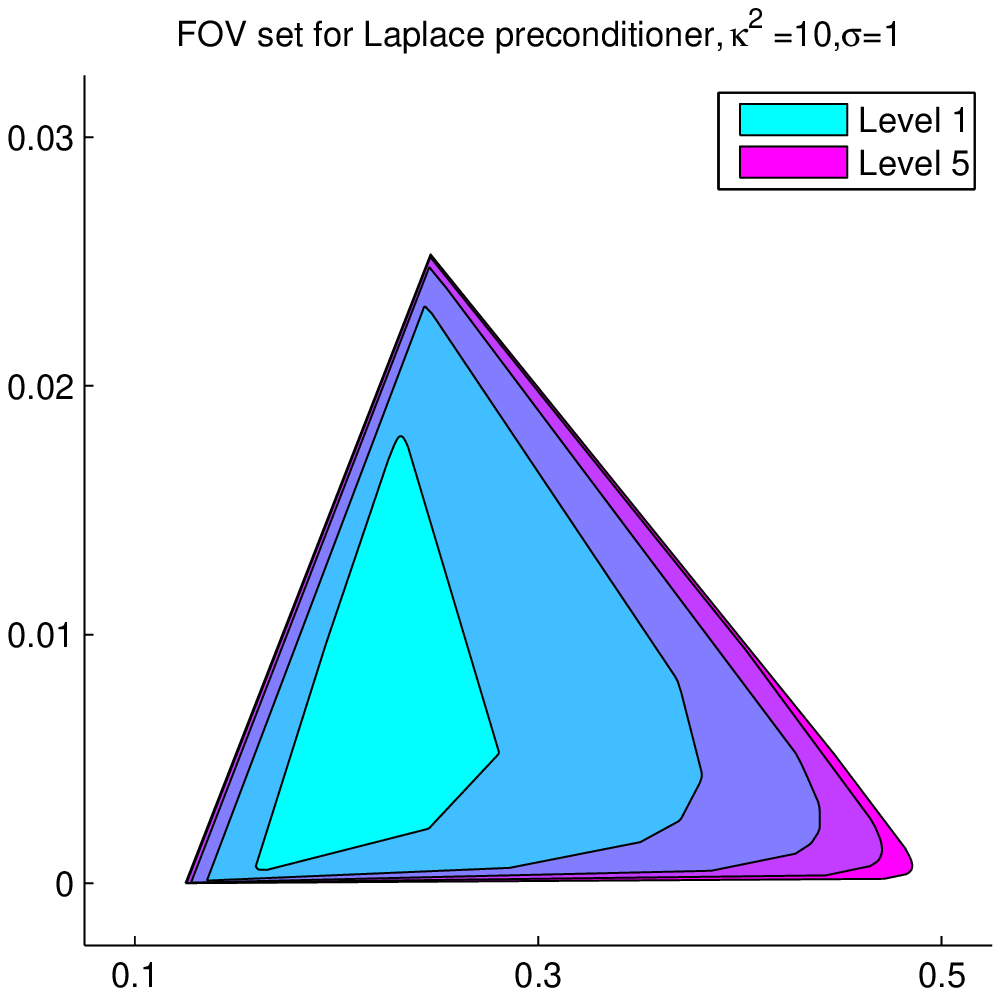} \\ 
\includegraphics[scale=0.5]{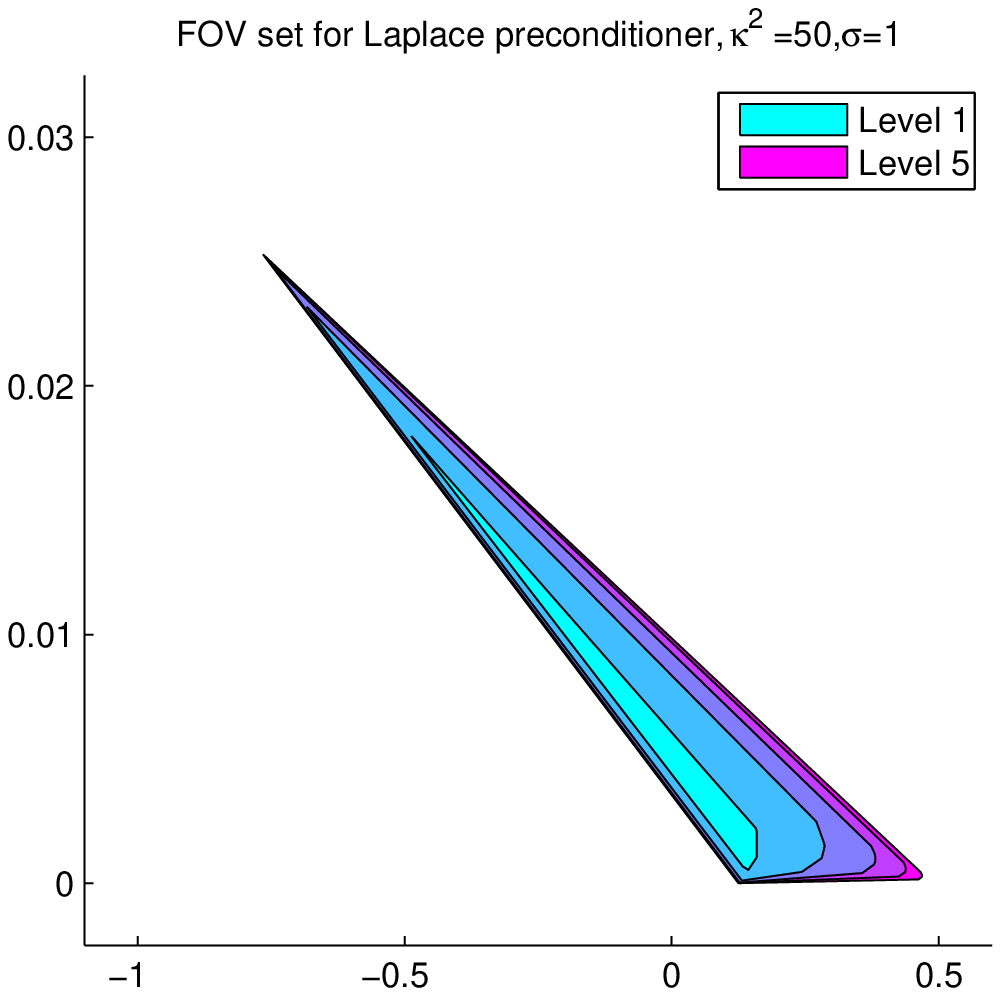} 
\end{center}
\caption{The $h^2$-scaled FOV for the exact Laplace preconditioner in the two dimensional test case. The sets are computed on mesh levels from one to five. The parameters are $\sigma =1$ and $\kappa^2 = 1, 10, 50$. }
\label{fig:test1_FOVsets}
\end{figure}

The $h^2$-scaled FOV for the Laplace preconditioned system are presented in Fig. \ref{fig:test1_FOVsets} for $\sigma = 1$ and three different parameter values $\kappa^2 = 1,10, 50$. Each figure contains the FOV from mesh levels one to five. Based on these results, the $h^{2}$-scaled shape of the FOV seems to be dependent only on the parameters $\kappa$ and $\sigma$.

To verify the theoretical bounds for the rectangle bounding the FOV set, we have computed the $h^{2}$ - scaled dimensions of this rectangle for mesh levels from one to four and for different parameter values. The bounds are presented in Figs. \ref{fig:test1_rectangle_sigma} and \ref{fig:test1_rectangle_kappa}. One can immediately observe from these results that the dimensions of the rectangle converge to a limit value when the mesh level is increased. The conclusion is that the $h^2$-scaled size of the rectangle stays constant, as predicted by the analysis. Based on Figs. \ref{fig:test1_rectangle_sigma} and \ref{fig:test1_rectangle_kappa}, the lower bound of the real part depends linearly on $\kappa^2$ and the upper bound of the imaginary part linearly on $\sigma$. The upper bound for the real part as well as the bounds for the imaginary part are $\kappa^2$ -independent. The bounds for the real part are $\sigma$-independent and the lower bound for the imaginary part is very close to zero. These results are in good accordance with Theorem \ref{th:L_rectangle}.

\begin{figure}[ht]
\begin{center}\hfill
\includegraphics[scale=0.5]{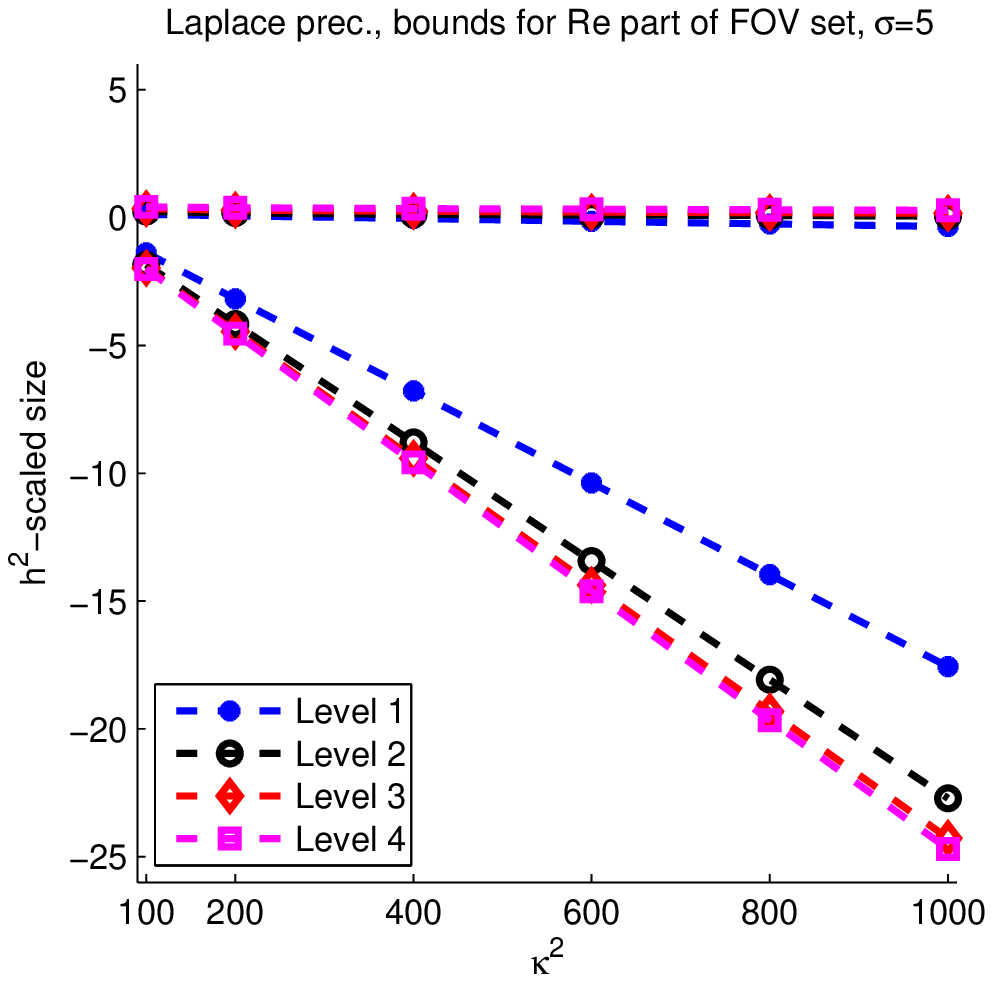} \hfill \includegraphics[scale=0.5]{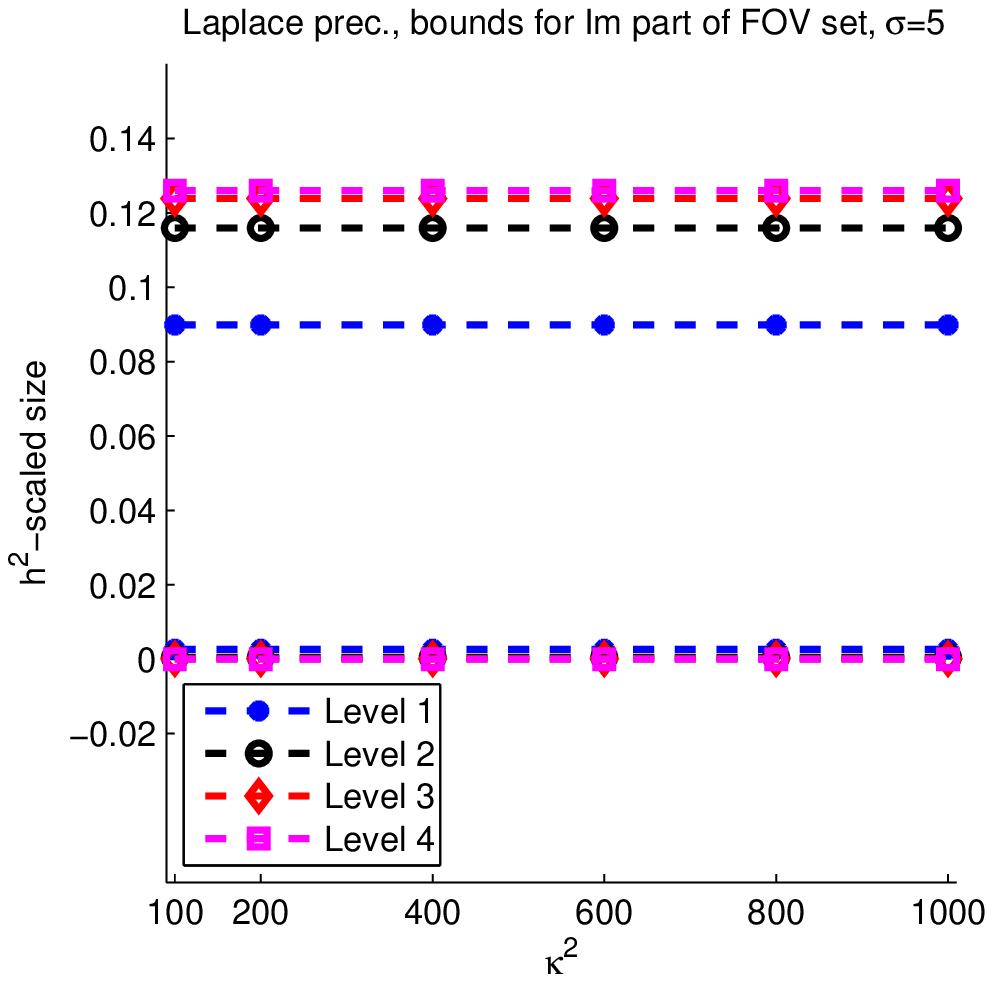} \hfill
\end{center}
\caption{The $h^2$-scaled bounds for the real and imaginary parts of the rectangle containing the FOV of the Laplace preconditioned system. The parameter $\sigma = 5$.}
\label{fig:test1_rectangle_sigma}
\end{figure}

\begin{figure}[ht]
\begin{center}\hfill
\includegraphics[scale=0.5]{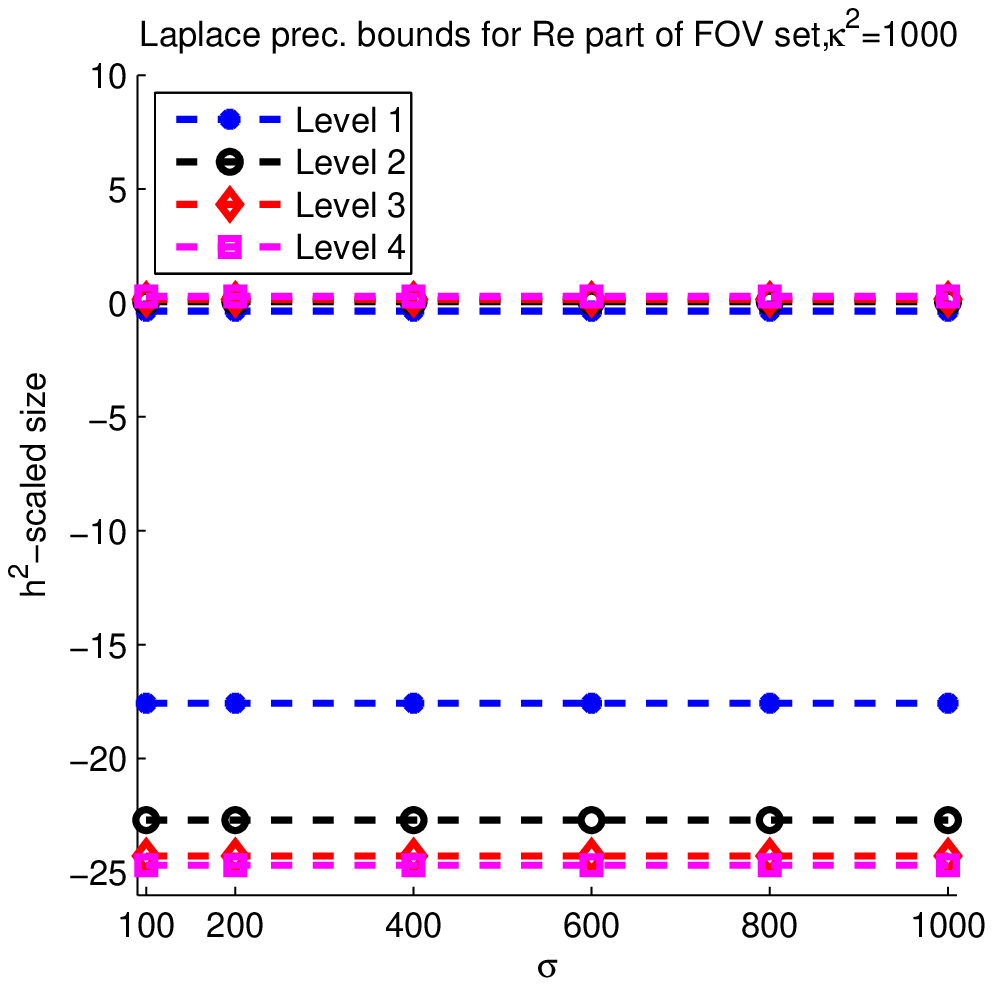} \hfill \includegraphics[scale=0.5]{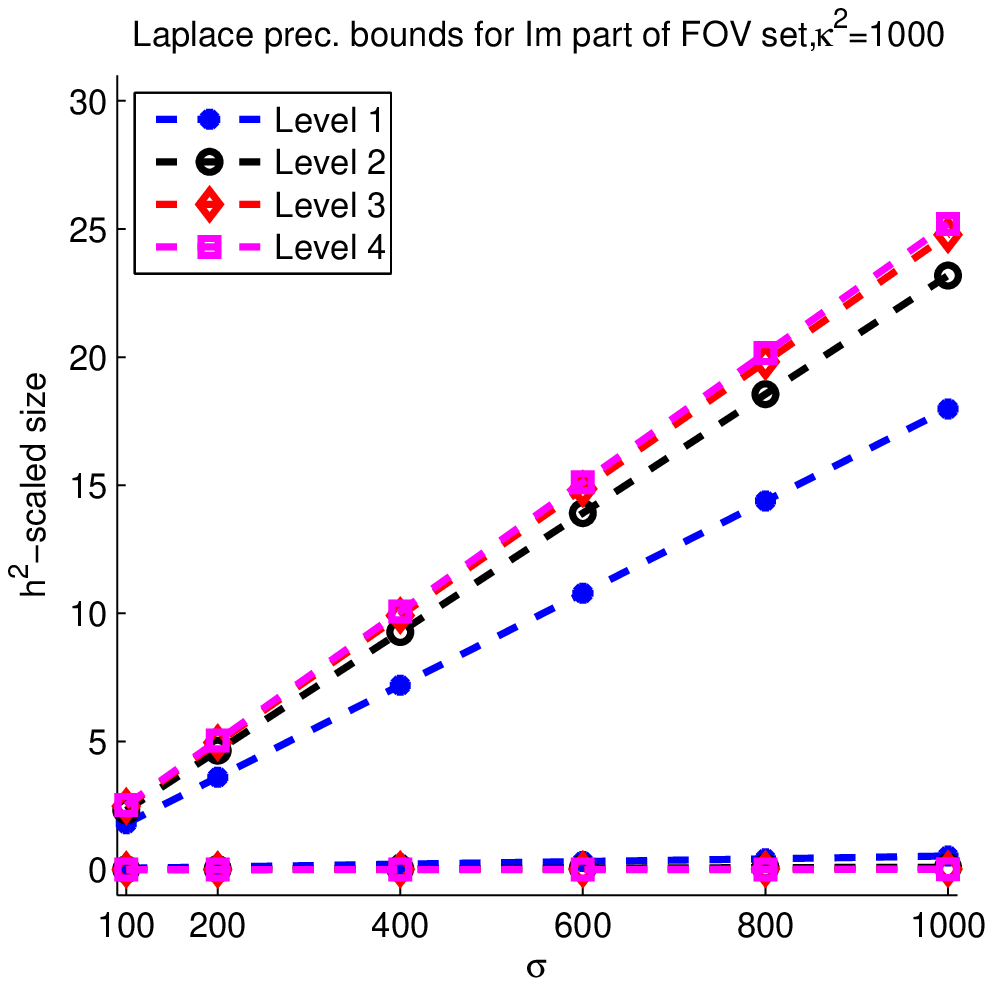}\hfill
\end{center}
\caption{The $h^2$-scaled bounds for the real and imaginary parts of the rectangle containing the FOV of the Laplace preconditioned system. The parameter $\kappa^2=1000$.}
\label{fig:test1_rectangle_kappa}
\end{figure}

Next, we consider the strip containing the FOV for different parameter values and mesh levels from one to four. The strip is computed by finding the two lines with the slope $\frac{ \kappa^2 }{\sigma}$ bounding the FOV from above and below. The bounds for the $x$-intercept of the strip (\ref{eq:num:strip}) are visualized in Fig. \ref{fig:test1_xintercept}. Based on these bounds, the $h^{2}$-scaled $x$-intercept points of the two lines bounding the FOV are independent of $\kappa$, $\sigma$ and $h$. This is as predicted by Theorem \ref{th:L_strip}.

\begin{figure}[ht]
\begin{center}
\includegraphics[scale=0.5]{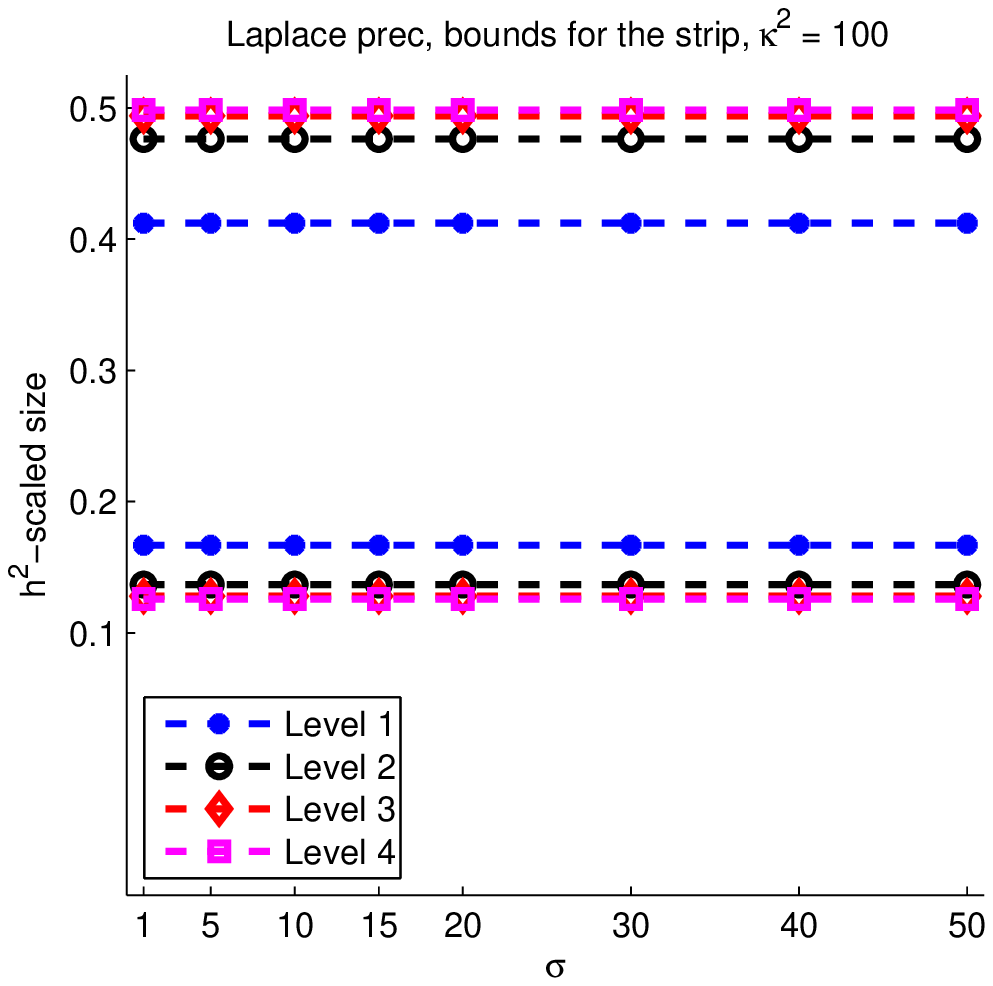}  \hfill
\includegraphics[scale=0.5]{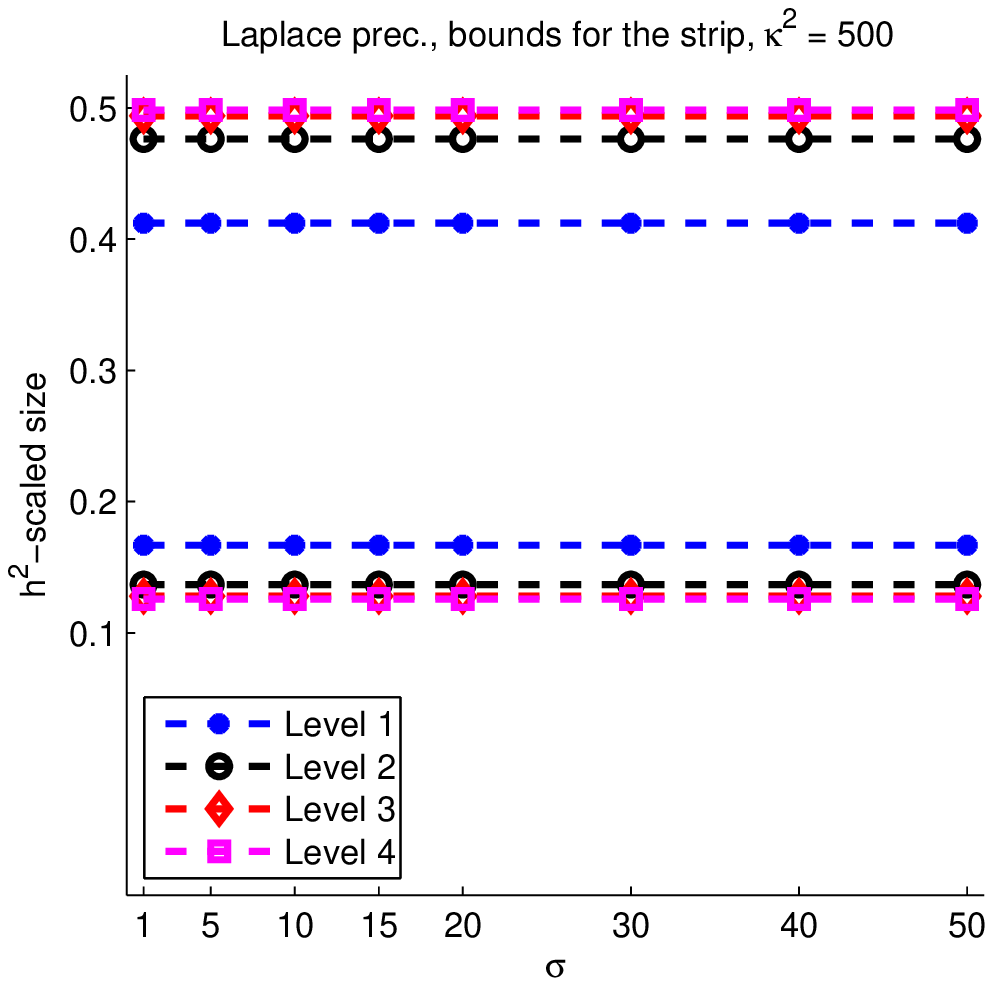}  \\
\includegraphics[scale=0.5]{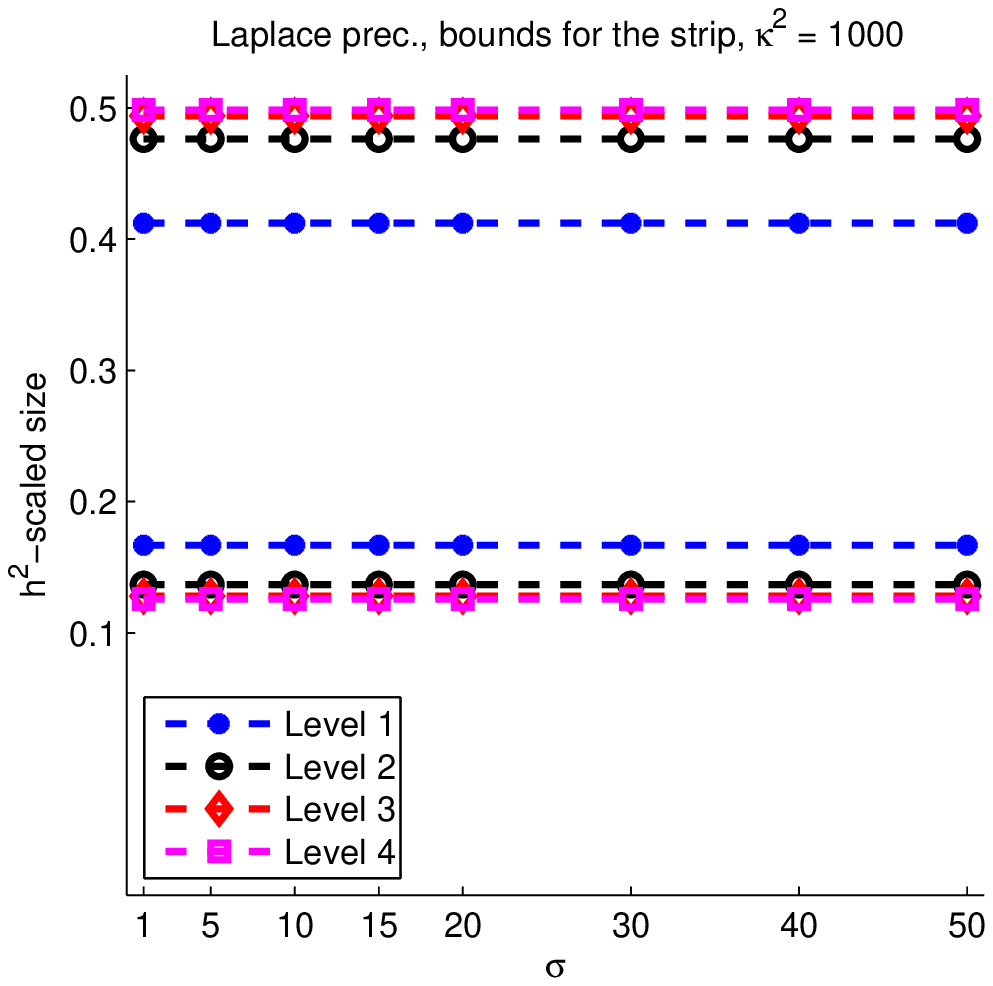}
\end{center}
\caption{The $h^2$-scaled bounds for $x$-intercepts of the strip containing the FOV for exact Laplace preconditioner in the first test case. The parameter $\kappa$ has the values $\kappa^2 = 100, 500, 1000$. }
\label{fig:test1_xintercept}
\end{figure}

\begin{figure}[ht]
\begin{center}
\includegraphics[scale=0.5]{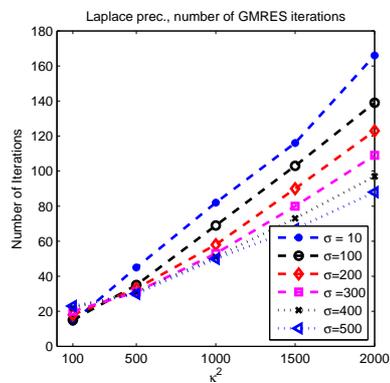}
\end{center}
\caption{The number of GMRES iterations required to solve the Laplace preconditioned system for $f=1$. A seventh level mesh and stopping criterion $10^{-6}$ were used in the test.}
\label{fig:test1_GMRES_iterations_laplace}
\end{figure}

To demonstrate the dependency of the required number of GMRES iterations for the Laplace preconditioned system on $\sigma$ and $\kappa$, we have solved the problem with different parameter values and the load function $f=1$. The level $7$ mesh was used in the computations. The stopping criterion for the GMRES iteration was set to $10^{-6}$. The required number of GMRES iterations is visualized in Fig. \ref{fig:test1_GMRES_iterations_laplace}. Based on these results, a linear dependency between the required number of iterations and $\kappa^2$ is observed. The slope of the $\kappa^2$ to number of iterations line is dependent on $\sigma$. This is due to the distance of the FOV from the origin being dependent of the ratio of $\kappa^2$ and $\sigma$.

\clearpage 
\subsection{Inexact Laplace preconditioner}

Next, we will consider replacing the exact Laplace preconditioner with a multigrid solver. The multigrid solver uses $V$-cycle iterations with one pre- and postsmoothing step with a Richardson smoother. A mesh hierarchy of mesh levels from two to seven is used in the tests, if not otherwise stated. Our aim is to demonstrate the bounds for the FOV of the MG preconditioned system presented in Section 5. The main interest lies in the behavior of the perturbation set

\begin{equation*}
\mathcal{F}( A ( \tilde{K}^{-N}-K^{-1} ) M)
\end{equation*}

\noindent when $\kappa$, $\sigma$, the number of multigrid iterations, or the number of levels are varied. In Section 5, we have shown that the perturbation set is located inside a $h^2$-scaled rectangle 

\begin{equation*}
(-c \gamma_0^N - c\kappa^2 \gamma_1^N , C\gamma_0^N + C\kappa^2 \gamma_1^N) \times (-c\gamma_0^N - c\sigma^2 \gamma_1^N , C\gamma_0^N + C\sigma^2 \gamma_1^N)
\end{equation*}

\noindent where $c,C > 0 $ are constants independent of the computational mesh, $\kappa$, $\sigma$, and the applied iterative scheme. The parameters  $\gamma_0$ and $\gamma_1$ are the error reduction factors for the multigrid solver and $N$ is the number of $V$-cycles.

\begin{figure}[ht]
\begin{center}
\includegraphics[scale=0.5]{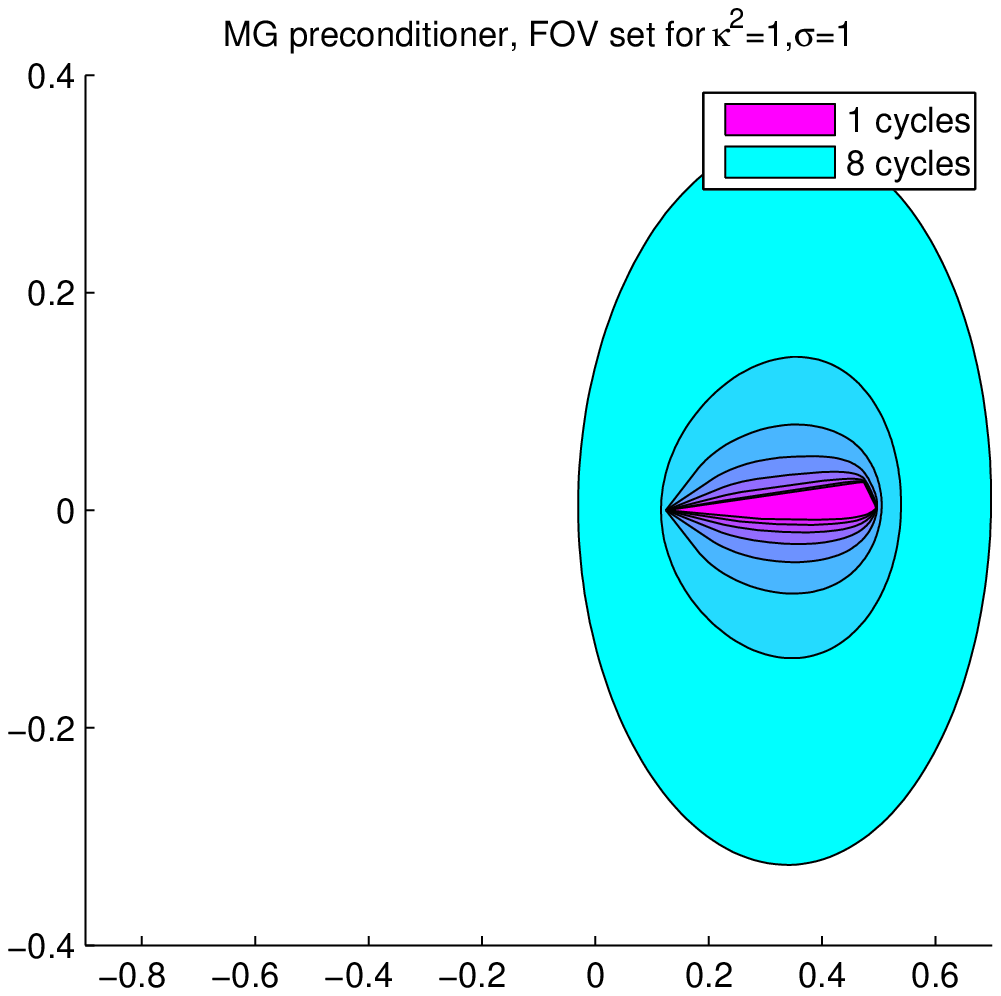} \hfill
\includegraphics[scale=0.5]{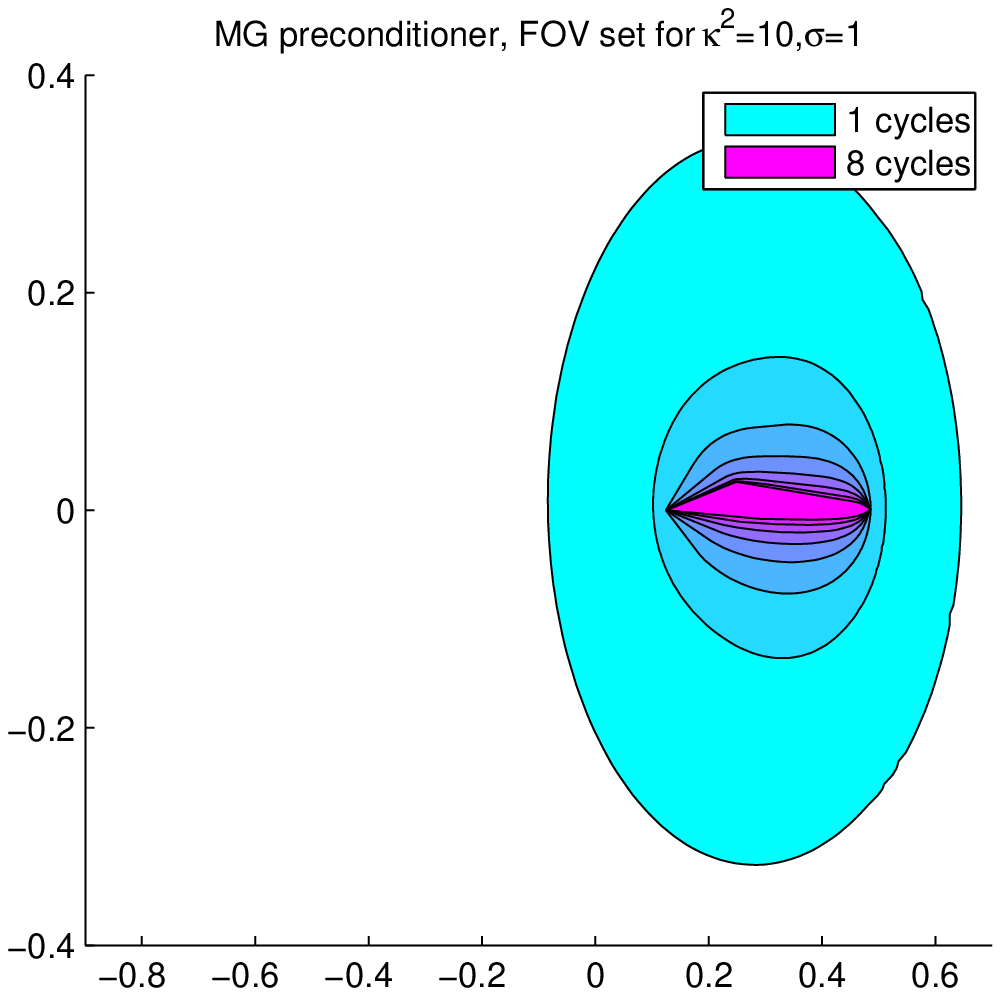} \\ 
\includegraphics[scale=0.5]{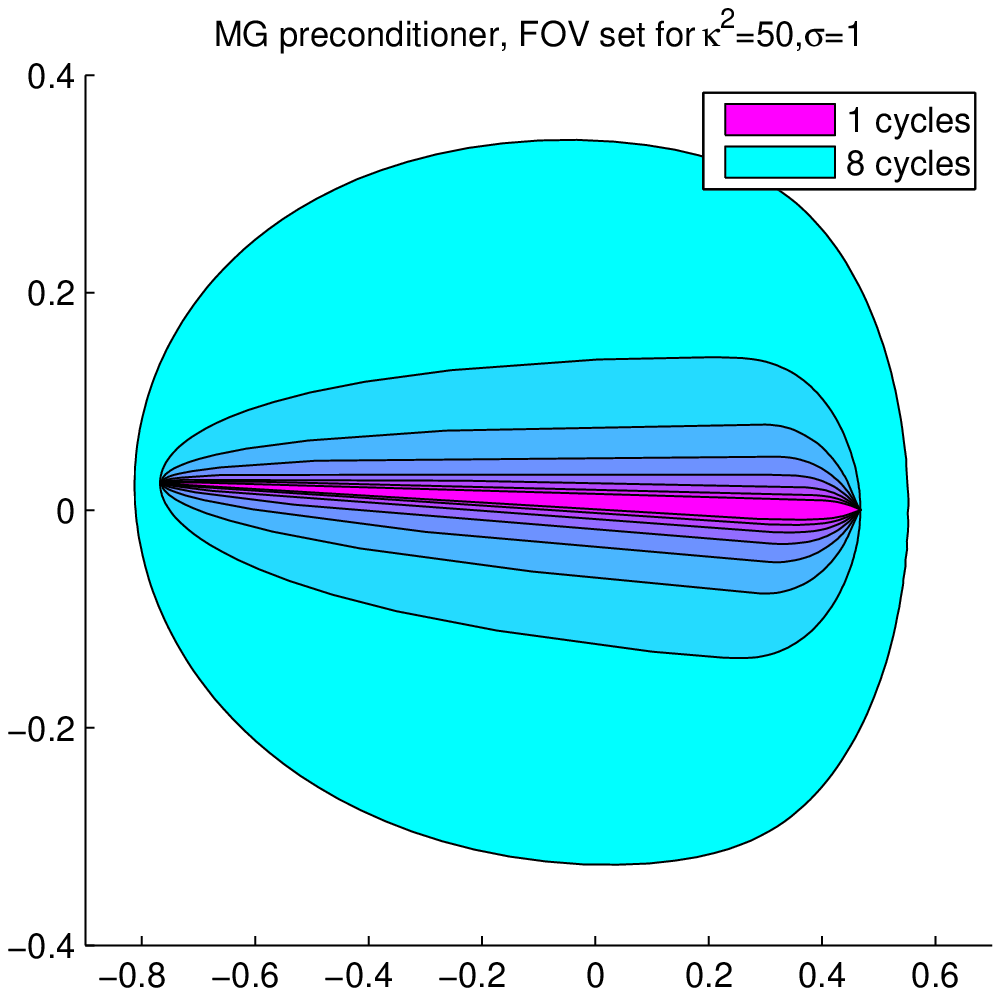} 
\end{center}
\caption{The $h^2$-scaled FOV for the multigrid preconditioned system. Up to $8$ multigrid $V$-cycles have been used. The parameter $\sigma = 1$ and $\kappa^2 = 1, 10, 50$.}
\label{fig:test1_mg_fovset}
\end{figure}

\begin{figure}[ht]
\begin{center}
\includegraphics[scale=0.5]{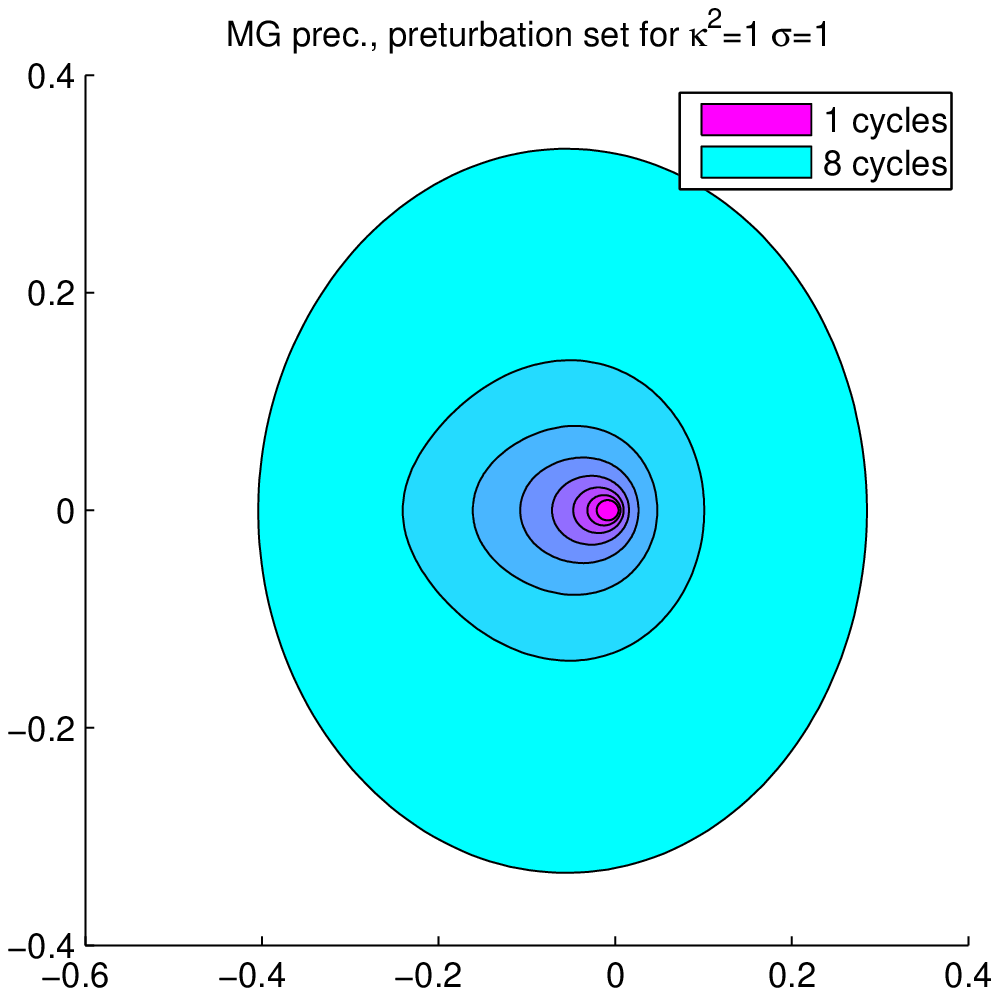} \hfill
\includegraphics[scale=0.5]{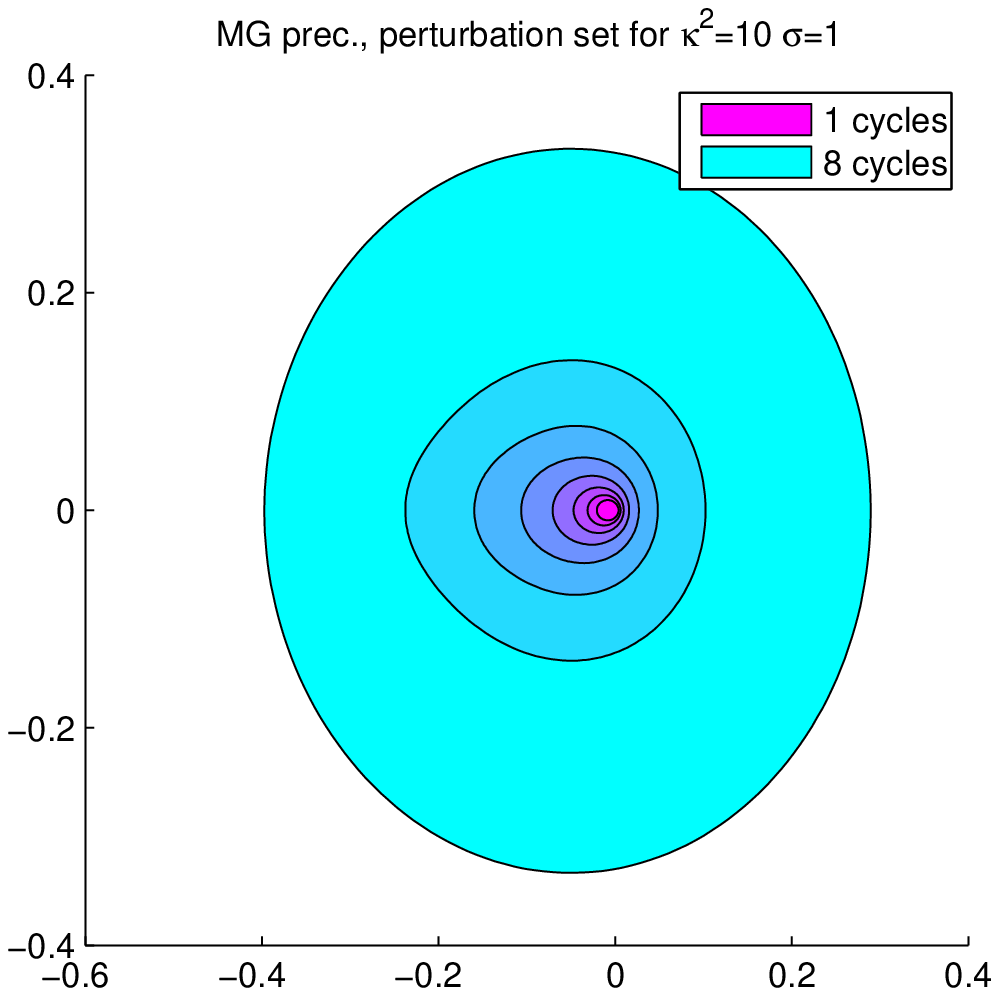} \\
\includegraphics[scale=0.5]{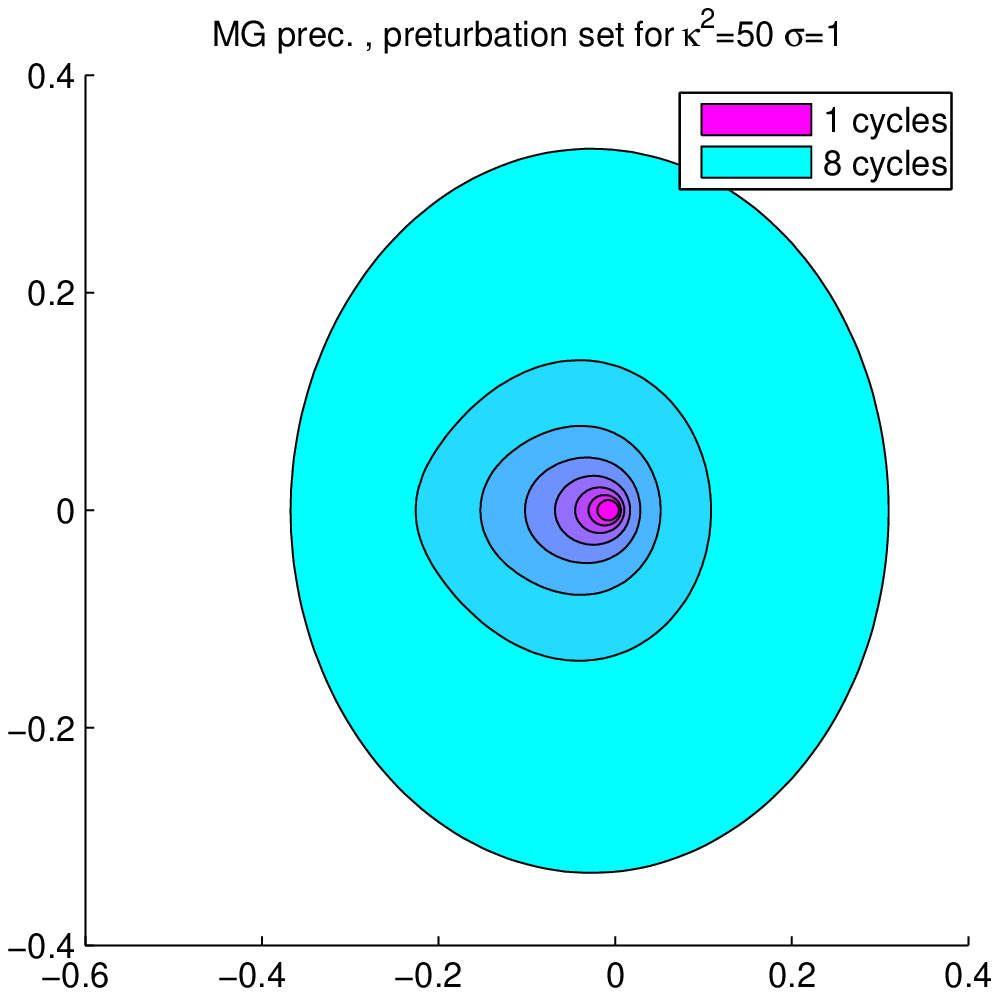}  
\end{center}
\caption{The $h^2$-scaled perturbation sets for the multigrid preconditioner. The parameter $\sigma = 1$ and $\kappa^2 = 1, 10, 50$. }
\label{fig:test1_mg_perturbationset}
\end{figure}

The FOV for the MG preconditioned system is presented for parameters $\kappa^2 = 1$, $10$, $50$, $\sigma = 1$, and a varying number of $V$-cycles in Fig. \ref{fig:test1_mg_fovset}. Based on this figure, one can verify that the FOV converges to a limit set, when the number of $V$-cycles in increased. This behavior is as predicted in Section 5. Examples of the corresponding perturbation sets are given in Figure \ref{fig:test1_mg_perturbationset}. The diameter of the perturbation sets clearly converges to zero when the number of $V$-cycles is increased. Due to the relatively small values of $\kappa$, all perturbation sets look rather similar.

\begin{figure}[ht]
\begin{center}
\includegraphics[scale=0.5]{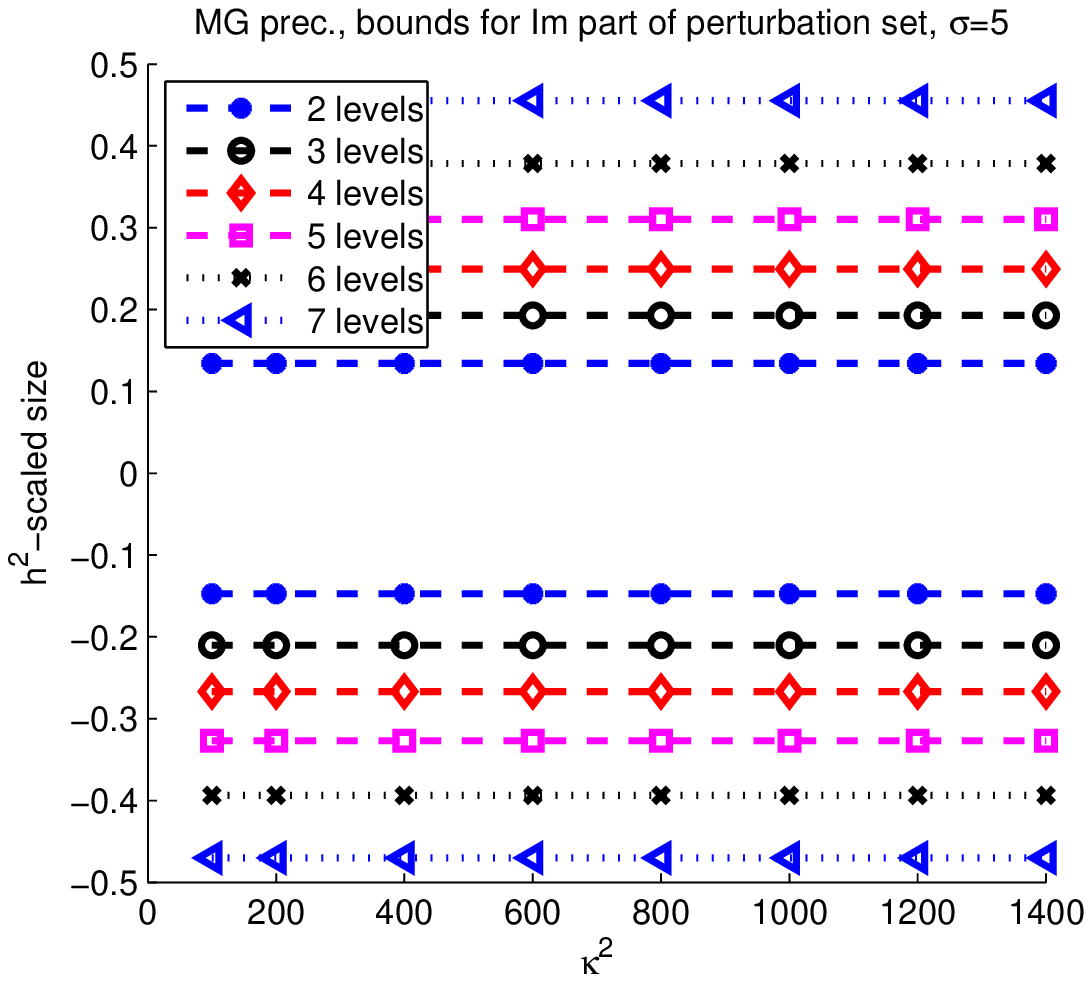} \hfill 
\includegraphics[scale=0.5]{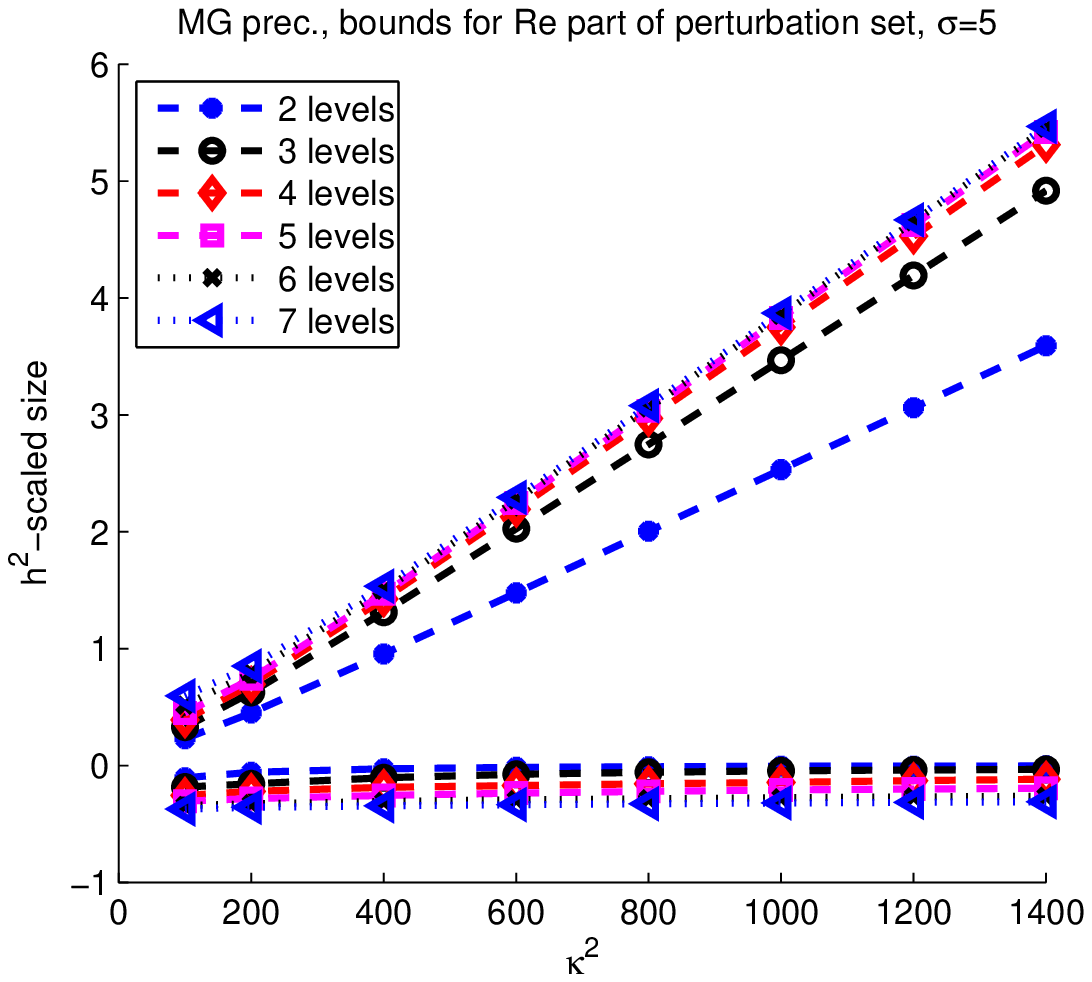}
\end{center}
\caption{The $h^2$-scaled bounds for the real and imaginary parts of the perturbation set for the MG preconditioner with a single $V$-cycle. A Level one mesh is used for the coarse grid and the number of levels in the multigrid hierarchy is varied from two to seven. The parameter is $\sigma = 5$.}
\label{fig:test1_mg_Vcycle_s5}
\end{figure}

Next, we will study the dependency of the dimensions of the perturbation set on the number of $V$-cycles and $h$. In this test, the coarse mesh is kept unchanged, but the number of levels in the multigrid hierarchy is varied from two to seven. The coarse grid for each test is the level two mesh. The results are presented in Figs. \ref{fig:test1_mg_Vcycle_s5} and \ref{fig:test1_mg_Vcycle_k1000}.

From Fig. \ref{fig:test1_mg_Vcycle_s5}, one can observe a linear dependency between $\kappa^2$ and the upper bound for the real part of the perturbation set. No $\kappa^2$-dependency is observed for the imaginary part of the set. The number of levels in the multigrid hierarchy affect both the slope and the $y$-intercept point of the computed lines. The bounds in Section 5 predict that the parameter $\gamma_0$ determines the $y$-intercept point and parameter $\gamma_1$ changes the slope of the lines. The error reduction factors $\gamma_0$ and $\gamma_1$ of the MG method depend on the number of levels, see e.g. \cite{Br:2007}, which explaines this phenomenon.

\begin{figure}[ht]

\begin{center}
\includegraphics[scale=0.5]{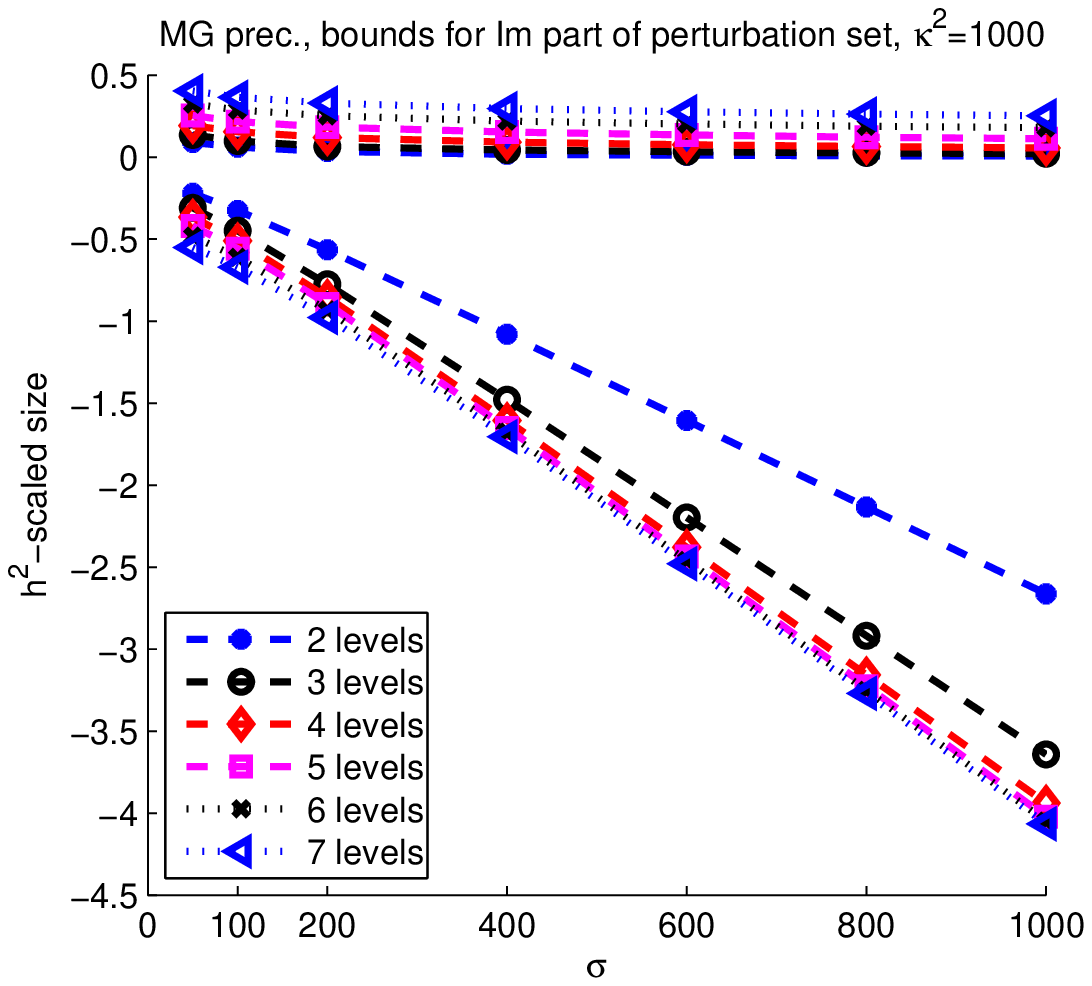} \hfill 
\includegraphics[scale=0.5]{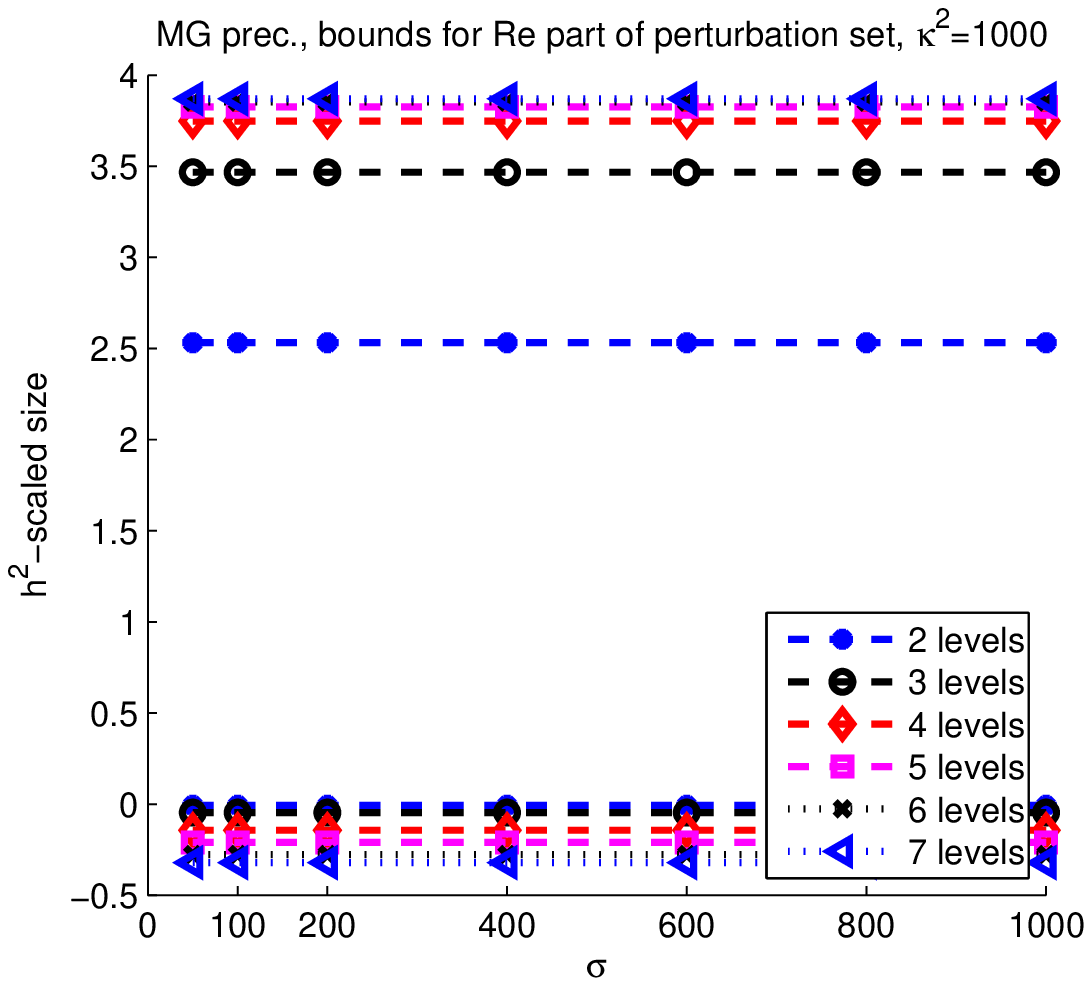}
\end{center}
\caption{The $h^2$-scaled bounds for the real and imaginary parts of the perturbation set for the MG preconditioner with a single $V$-cycle.  A level two mesh is used for the coarse grid and the number of levels in the multigrid hierarchy is varied from two to seven. The parameter is $\kappa^2 = 1000$.}
\label{fig:test1_mg_Vcycle_k1000} 
\end{figure}

The results for varying $\sigma$ in Fig. \ref{fig:test1_mg_Vcycle_k1000} are very similar. The lower bound for the imaginary part of the perturbation set depends linearly on $\sigma$, whereas the real part is $\sigma$-independent. Again, the intercept points and slopes of the lines change due to the changing error reduction factors. 

The results on the size of the perturbation set are in accordance with the derived bounds. However, the estimate for the upper bound of the real part and lower bound for the imaginary part seem to stay constant for varying $\kappa$ and $\sigma$. The bounds given in Section 5 overestimate the size of the set by predicting $\kappa^2$- and $\sigma$-dependency in these cases. As the FOV is a direct sum of the perturbation set and the FOV set for the Laplace preconditioner, the derived bounds manage capture the behavior of the FOV regardless of this overestimation.

Finally, we study how the dimensions of the perturbation set behave when the number of $V$-cycles grows. We have used a hierarchy with four levels and varied the number of $V$-cycles from one to four. The dependency of the bounds for the real part on $\kappa$ and the number of $V$-cycles
is visualized in Fig. \ref{fig:test1_mg_k_Vcycle2} and the dependency of the imaginary part on $\sigma$ and the number of $V$-cycles in Fig. \ref{fig:test1_mg_s_Vcycle2}. In addition, we have visualized the development of the bounds for parameters $\kappa^2 =1000$ and $\sigma=5$ for larger number of $V$-cycles in Fig. \ref{fig:test1_mg_perturbation_loglog}.

\begin{figure}
\begin{center}\hfill
\includegraphics[scale=0.5]{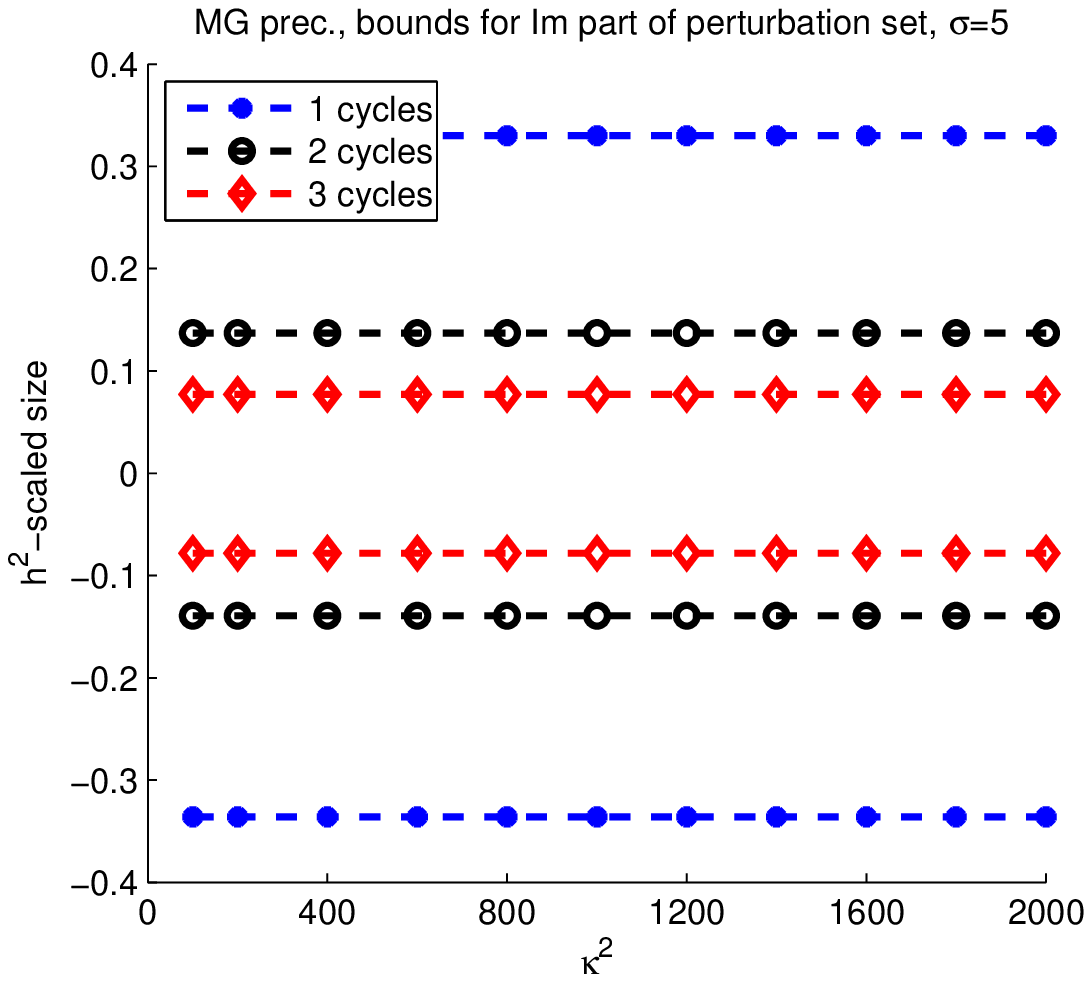} \hfill 
\includegraphics[scale=0.5]{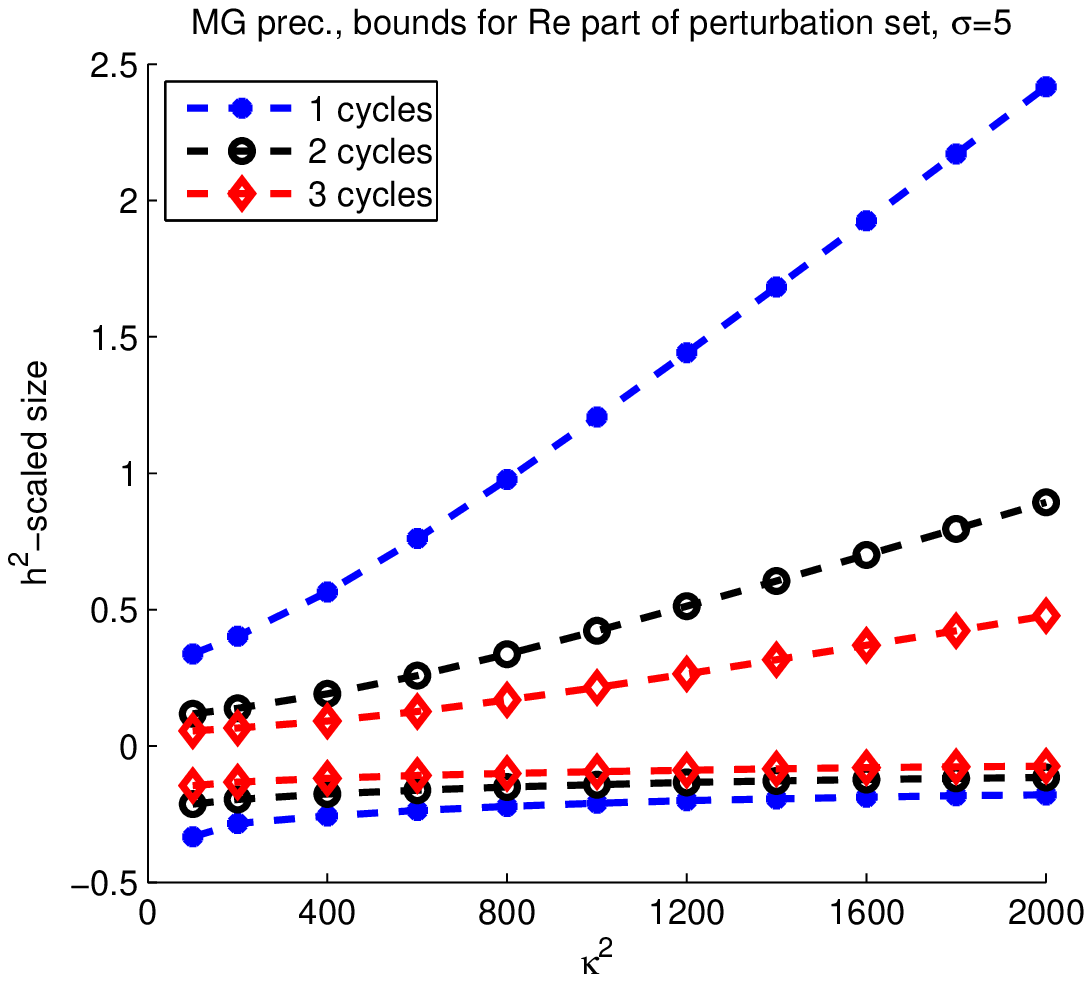}\hfill
\end{center}
\caption{The $h^2$-scaled bounds for the real and imaginary parts of the perturbation set for the MG preconditioner. The coarse grid and the number levels in the multigrid hierarchy are fixed and the number of $V$-cycles is varied. The parameter is $\sigma = 5$.}
\label{fig:test1_mg_k_Vcycle2}
\end{figure}

\begin{figure}
\begin{center}\hfill
\includegraphics[scale=0.5]{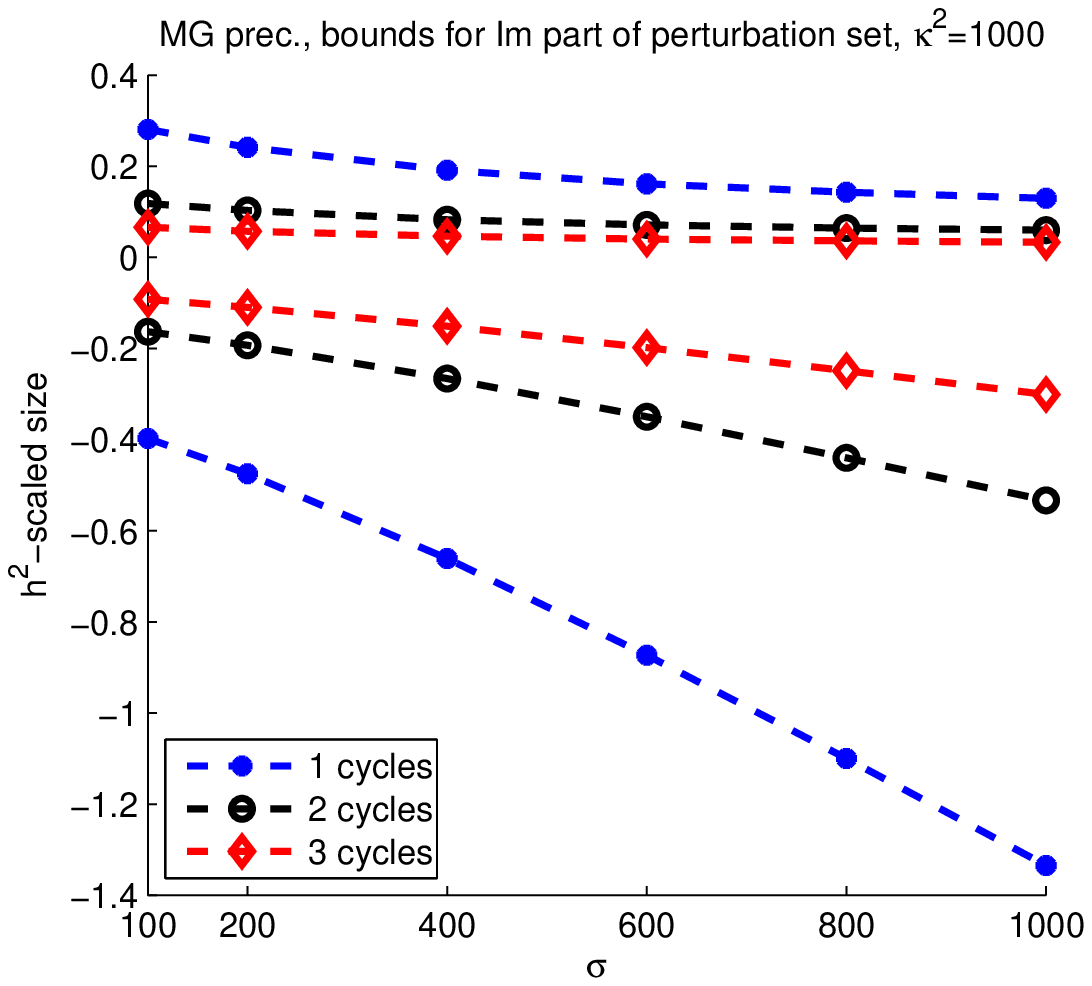} \hfill 
\includegraphics[scale=0.5]{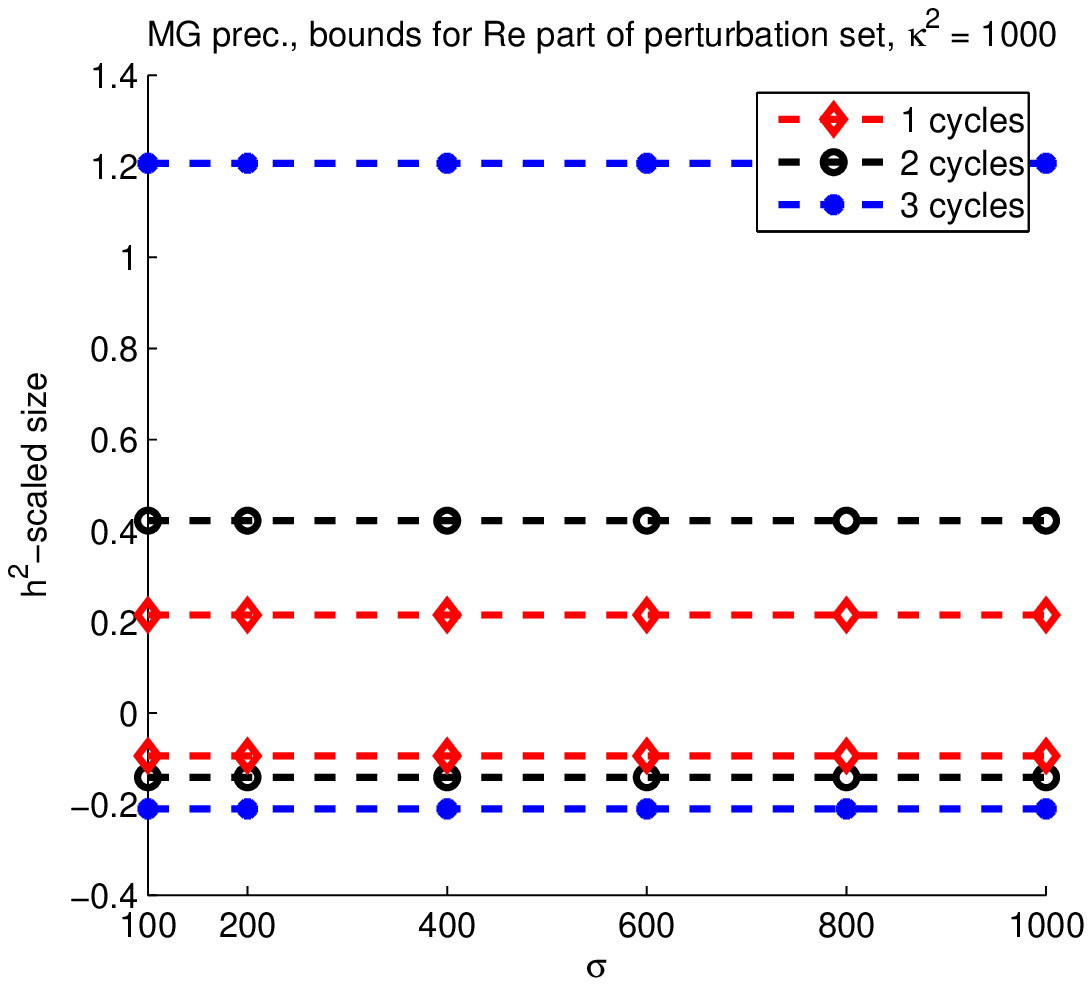}\hfill
\end{center}
\caption{The $h^2$-scaled bounds for the real and imaginary parts of the perturbation set for the MG preconditioner. The coarse grid and the number levels in the multigrid hierarchy are fixed and the number of $V$-cycles is varied. The parameter is $\kappa^2 = 1000$.}
\label{fig:test1_mg_s_Vcycle2}
\end{figure}

\begin{figure}
\begin{center}
\includegraphics[scale=0.5]{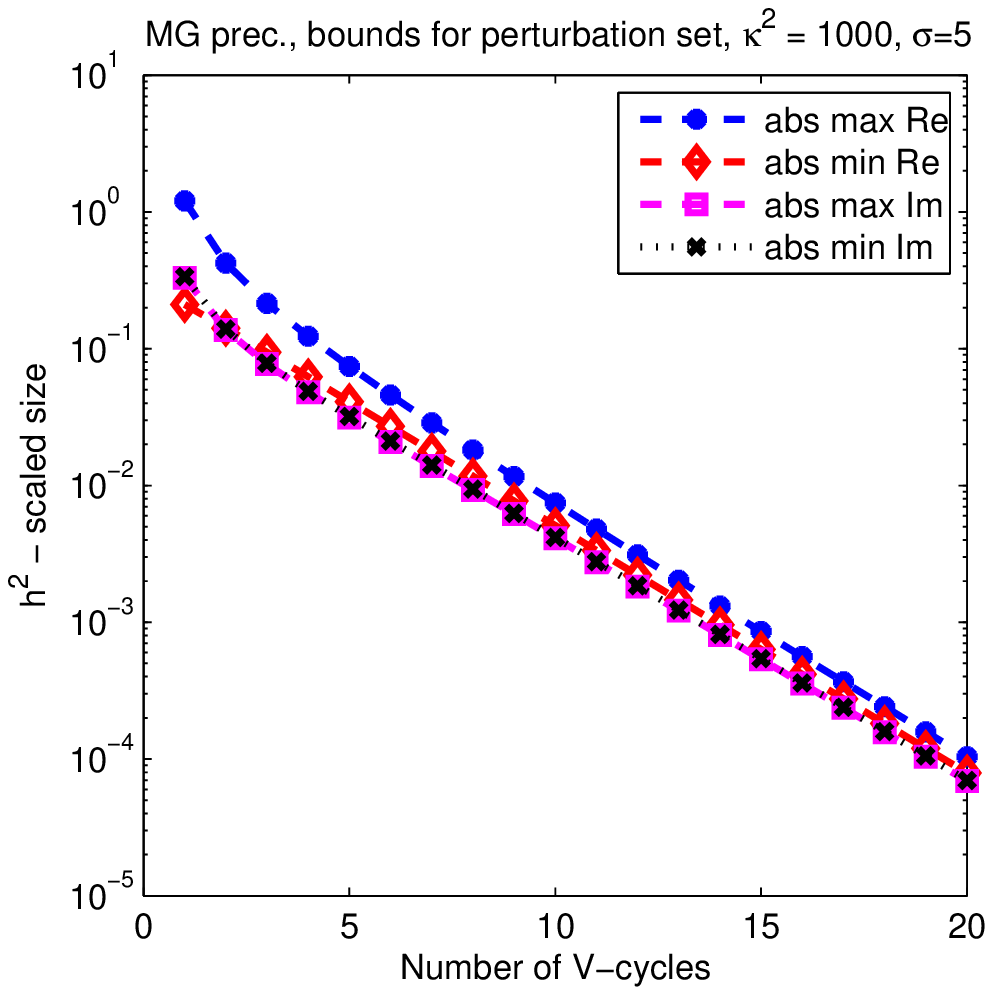}
\end{center}
\caption{Convergence of the $h^2$-scaled bounds for the real and imaginary parts of the perturbation set for the MG preconditioner as a function of the number of $V$-cycles. The parameters are $\sigma = 5$ and $\kappa^2=1000$}
\label{fig:test1_mg_perturbation_loglog}
\end{figure}

Based on these results, the size of the perturbation set decreases as predicted when the number of $V$-cycles is increased. The slope and the intercept points of the lines tend to zero. Based on Fig. \ref{fig:test1_mg_perturbation_loglog}, the convergence speed is same for all dimensions of the set, as predicted by in Section 5.

To study the dependency of the required number of MG preconditioned GMRES iterations on the parameters and the number of $V$-cycles, we consider the problem with $f=1$. The MG method uses a mesh hierarchy with six levels, with the level two mesh acting as the coarse grid. The stopping criterion for the GMRES iteration was set to $10^{-6}$. The required number of GMRES iterations for different parameter values and different number of $V$-cycles is visualized in Fig. \ref{fig:test1_GMRES_iterations_MG}. 

The interesting factor for the multigrid preconditioner is the dependency of the required number of GMRES iterations on the number of $V$-cycles. This is demonstrated in Fig. \ref{fig:test1_GMRES_iterations_MG} by comparing the number of iterations for varying number of multigrid $V$-cycles against one using ten $V$-cycles. Based on these results, one can observe a clear dependency between the number $V$-cycles and the number of required iterations. For large values of $\kappa$, the number of $V$-cycles has a quite big effect to number of required number ofG MRES iterations. This is due to the $\kappa^2$-dependency of the size of the perturbation set. The inclusion of the origin to the FOV for the MG preconditioned system seems to be quite irrelevant for the convergence.

\begin{figure}
\begin{center}\hfill
\includegraphics[scale=0.5]{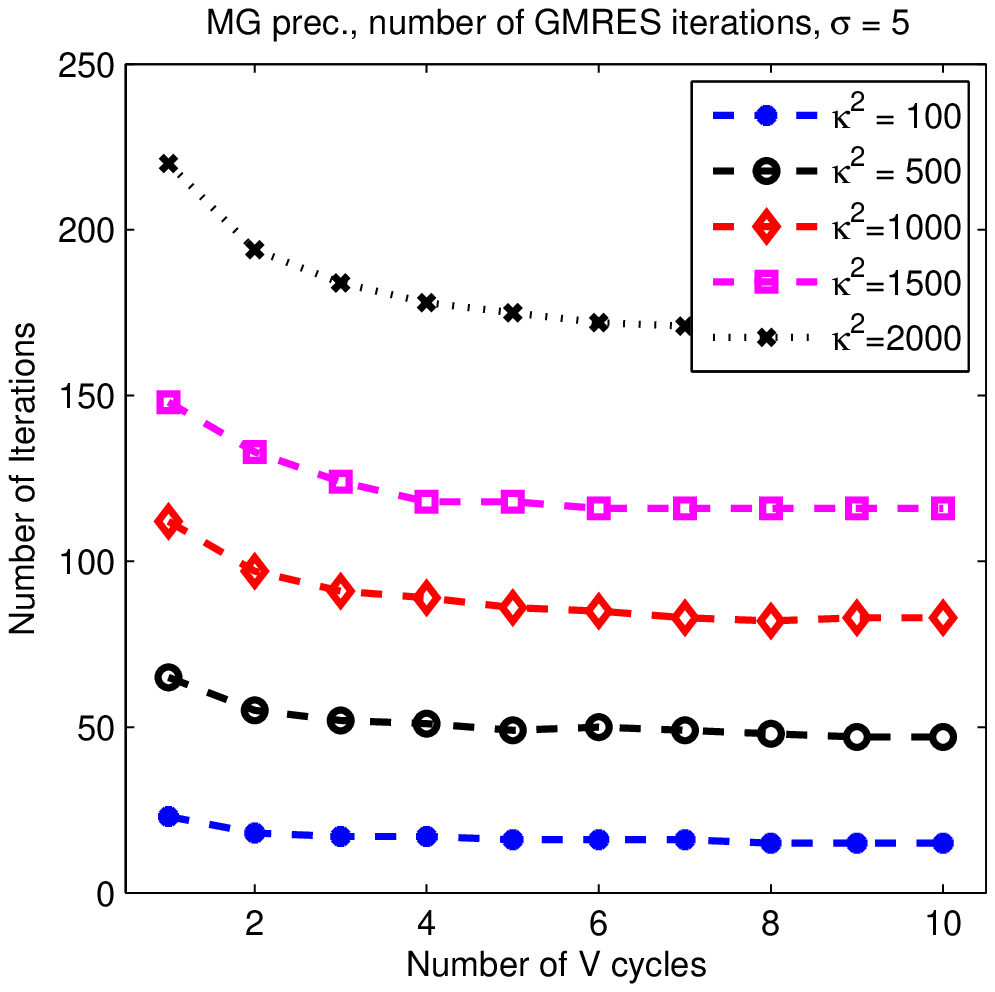} \hfill \includegraphics[scale=0.5]{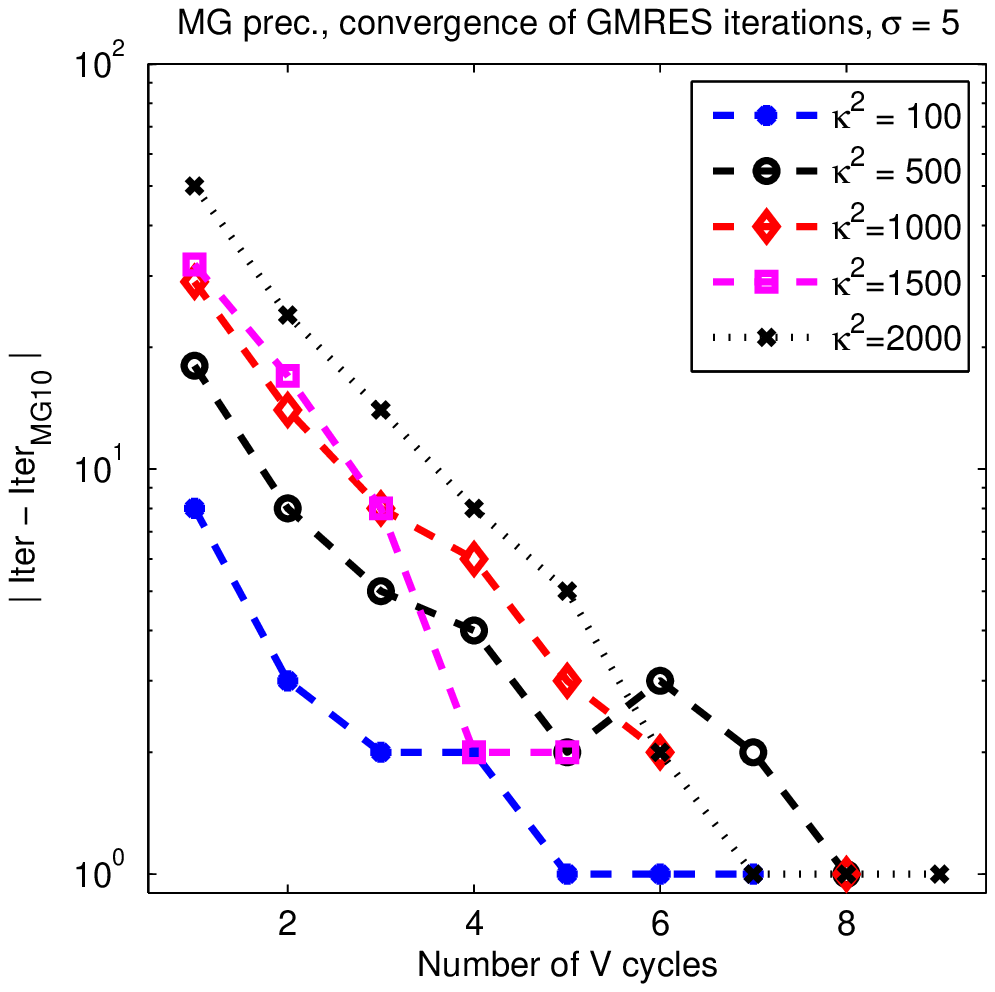}\hfill
\end{center}
\caption{The left figure shows the number of GMRES iterations required to solve the MG preconditioned system for $f=1$. The right figure shows the convergence of the number of GMRES iterations by comparing number of iterations for varying number of $V$-cycles to MG preconditioner using ten $V$-cycles.}
\label{fig:test1_GMRES_iterations_MG} 
\end{figure}

\clearpage
\subsection{Two-level preconditioner}

Finally, we consider the two-level preconditioner presented in Section 6. Our aim is to verify the bounds given in Theorem \ref{th:2L_bounds}. The most interesting parameter for the two-level preconditioner is the mesh size of the coarse grid. This parameter effectively determines the amount of computational work required to solve the linear system. 

In Fig. \ref{fig:test1_2L_FOV}, the FOV is computed for parameters $\kappa = 4\pi$ and $\sigma = 7$ by keeping the fine grid level fixed to seven and varying the coarse grid level. Based on these results, if a sufficiently small coarse grid mesh size is used the FOV is located at the right half-plane and does not contain the origin. In addition, the imaginary part of the FOV converges to zero. This behavior is as predicted in Section 6.

\begin{figure}
\begin{center}\hfill
\includegraphics[scale=0.5]{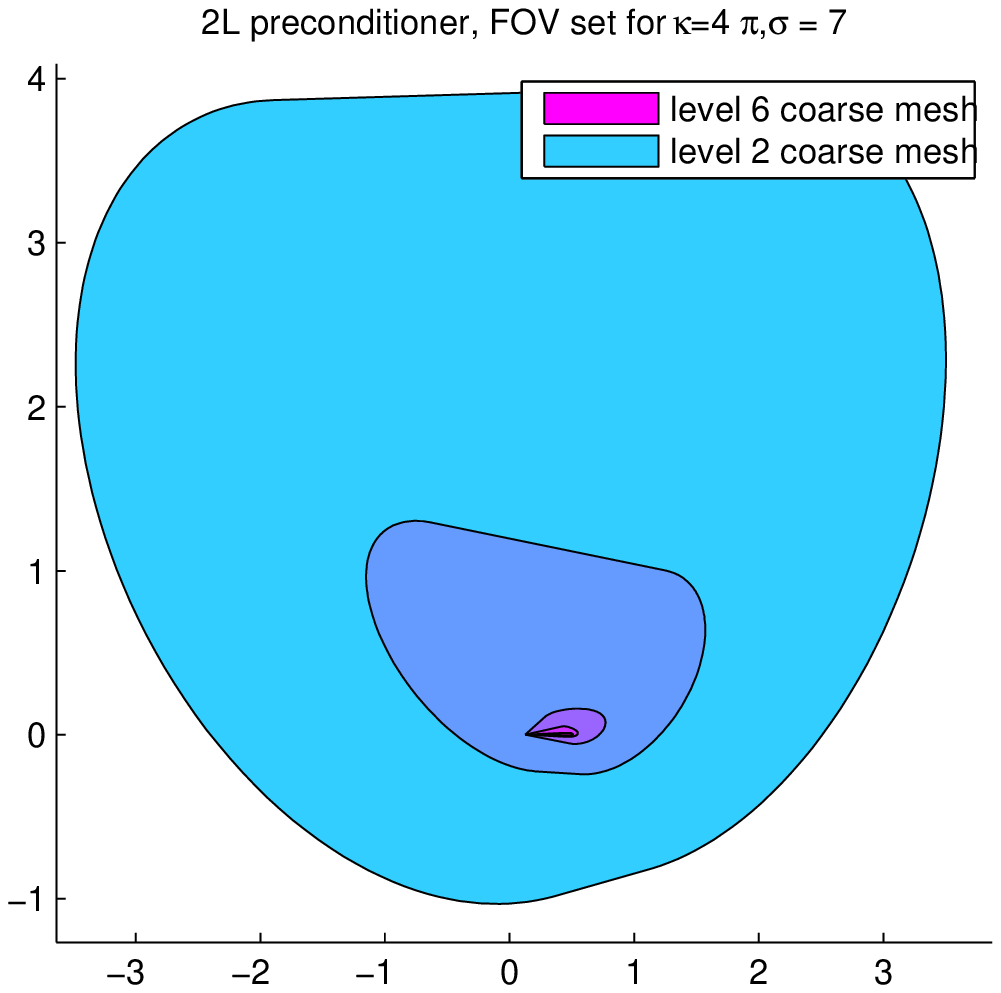} \hfill
\includegraphics[scale=0.5]{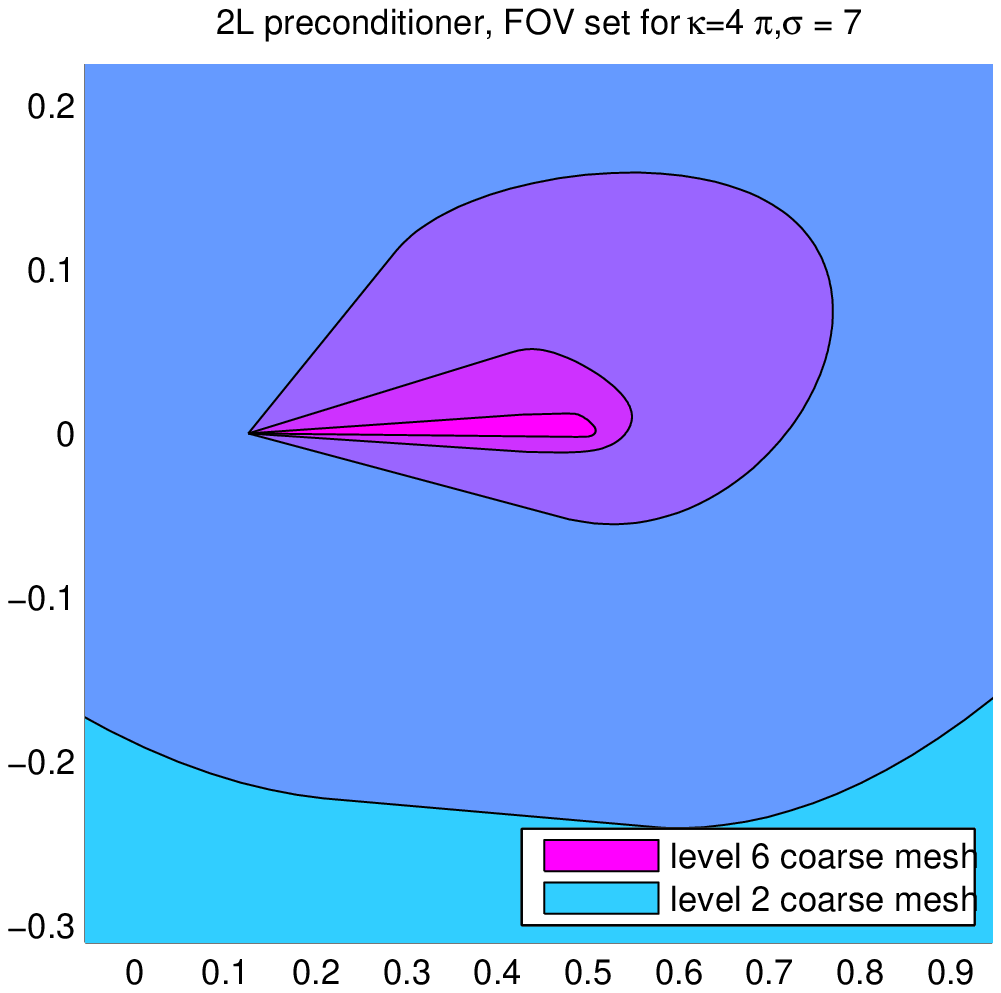}  \hfill
\end{center}
\caption{The $h^2$-scaled FOV sets for the two level preconditioned system. A seventh level fine mesh was used.}
\label{fig:test1_2L_FOV}
\end{figure}

The interesting result of Section 6 is the implication that the coarse grid mesh size in a convex domain should be such that the term $\kappa^3 H$ is small. We have computed the smallest real part of the FOV set for several different values of $\kappa$ using a level nine computational grid and different coarse grid levels. The coarse grid level required before FOV is located at the right half-plane is presented in the Fig.  \ref{fig:test1_2L_Nref}. As only few datapoints were computed it is difficult to determine if the requiremenet for the coarse grid mesh size is necessary or not. Unfortunatelly the FOV sets are computationally expensive to find, so we are not able to give a conclusive example.

\begin{figure}
\begin{center}
\includegraphics[scale=0.5]{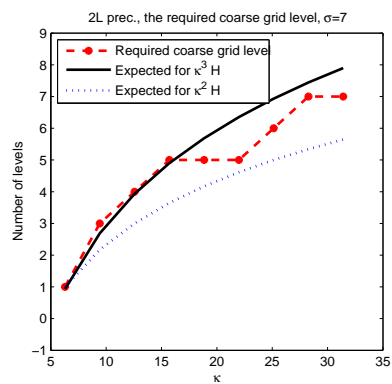}
\end{center}
\caption{The coarse mesh level required for the real part of the FOV set to be located in the right half plane. A level nine fine grid was used.}
\label{fig:test1_2L_Nref}
\end{figure}

In Fig. \ref{fig:test1_2L_conv_bounds} the behavior of the rectangle bounding the FOV of the two-level preconditioned system for parameters $\sigma=7$ and $\kappa=4\pi$, $6\pi$, $10\pi$ is visualized  as a function of the coarse grid mesh size $H$. Based on these results, after a sufficiently dense coarse mesh is reached the bounds for the imaginary part converge to zero. The bounds for the real part on the other hand converge to limit values. These results are in accordance with the behavior predicted for FOV in Section 6.  

\begin{figure}
\begin{center}
\includegraphics[scale=0.5]{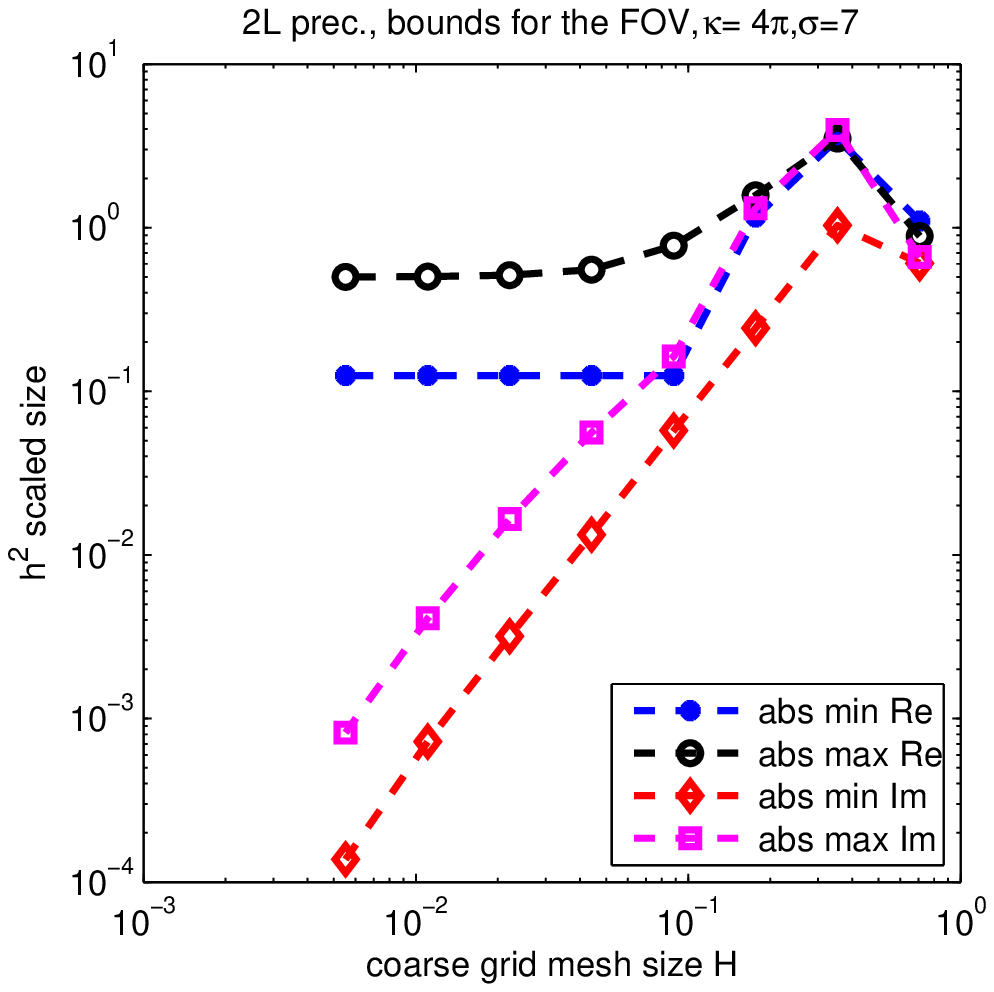} 
\includegraphics[scale=0.5]{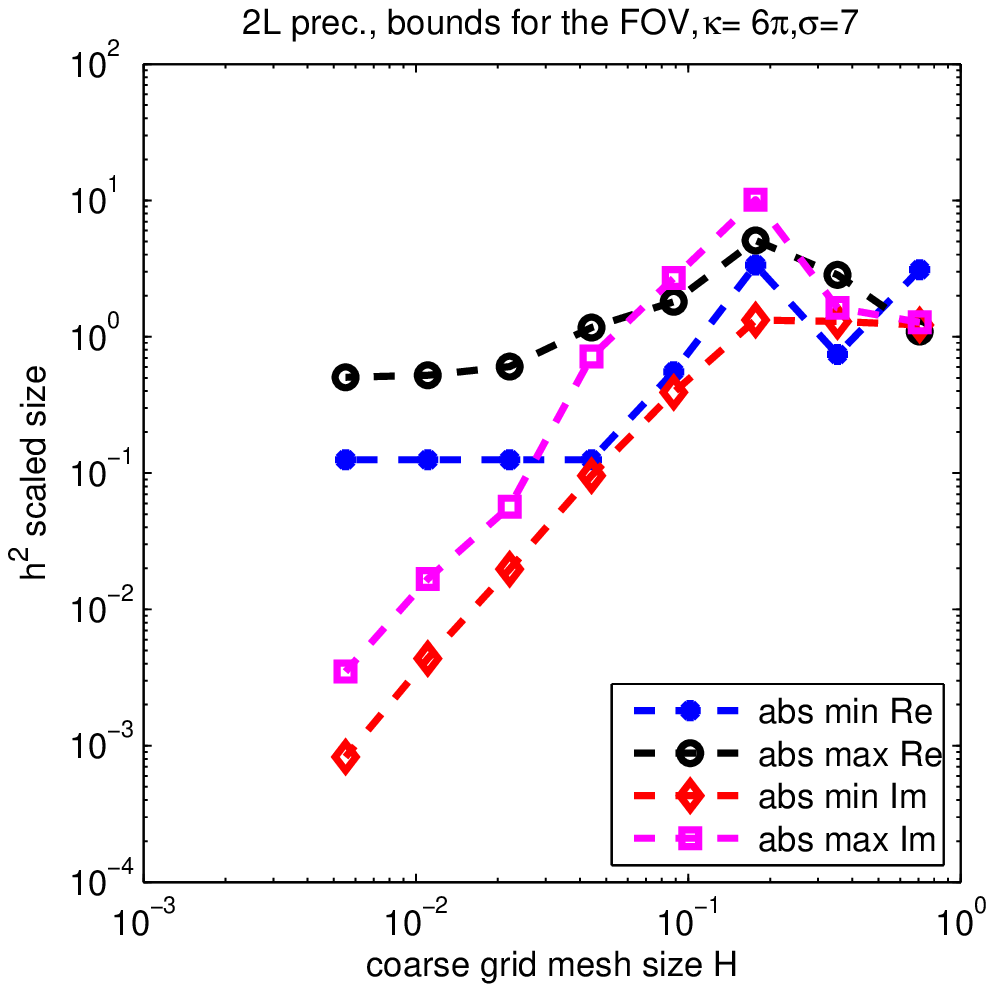} \\ 
\includegraphics[scale=0.5]{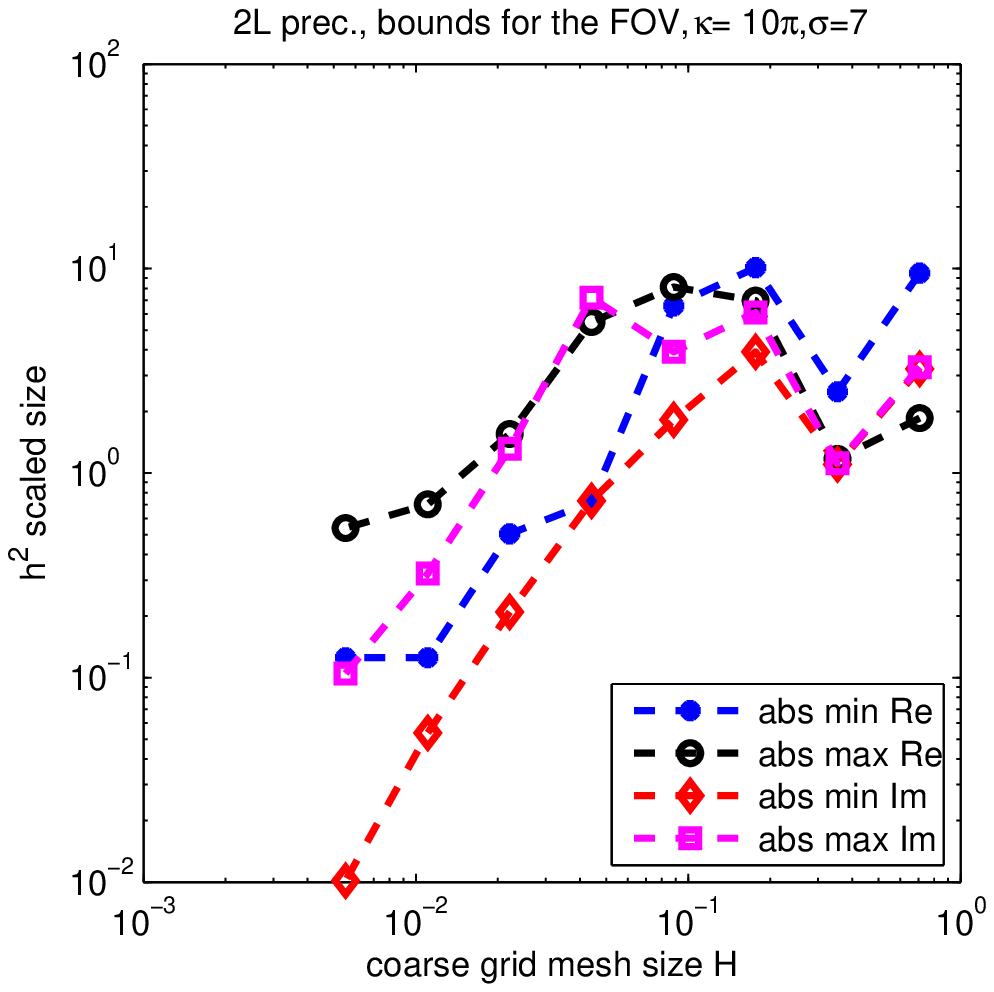} 
\end{center}
\caption{The $h^2$-scaled FOV sets for the two level preconditioner. The parameter $\sigma = 7$ and $\kappa = 4\pi, 6\pi , 10\pi$. }
\label{fig:test1_2L_conv_bounds}
\end{figure}

To study the two-level preconditioned GMRES method, we solve the linear system using a level nine fine grid and varying coarse grid level. The load is $f=1$ and the stopping criterion is set to $10^{-6}$. The results are presented in Fig. \ref{fig:test1_2L_GMRES}. These results are as predicted in Section 6. After a sufficiently large coarse grid level, the number of iterations stagnates to a limit value. The stagnation point depends on the parameter $\kappa$. An interesting observation is that the number of iterations converges regardless of the inclusion of origin in the FOV. 

\begin{figure}
\begin{center}
\includegraphics[scale=0.5]{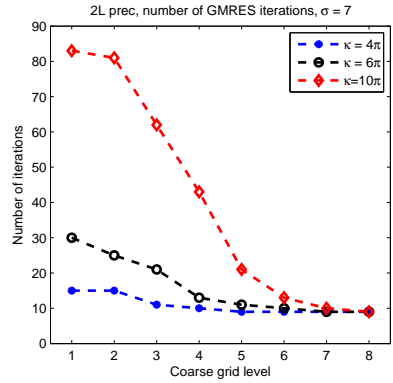}
\end{center}
\caption{The number of GMRES iterations required to solve the two-level preconditioned system. A level nine fine grid was used and the coarse grid level was varied.}
\label{fig:test1_2L_GMRES}
\end{figure}

\subsection{A three dimensional example}

We conclude the numerical examples by considering a three dimensional cube, $\Omega = (0,1)^3$. A hierarchy of uniformly refined tetrahedral meshes with five levels has been used in all of the tests. The finest level mesh had approximately $750 \cdot 10^3$ degrees of freedom and $4.5 \cdot 10^6$ tetrahedral elements. As we have demonstrated the behavior of the FOV in detail in the two dimensional test case, our interest will be solely on the number of iterations required to solve the preconditioned systems using GMRES. The load for all three dimensional test cases is chosen as $f=1$ and the stopping criterion for GMRES is chosen as $10^{-6}$.

We begin our experiment with the exact Laplace preconditioner. To solve the Laplace problem, a preconditioned conjugate gradient (PCG) method with a single multigrid $V$-cycle as a preconditioner was used. The stopping criterion for the PCG method is set to $10^{-12}$. The parameters $\sigma$ and $\kappa$ are both varied. 

\begin{figure}
\begin{center}
\includegraphics[scale=0.5]{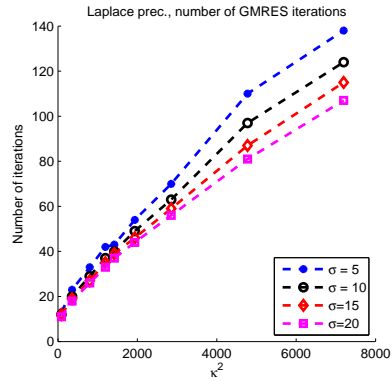}
\end{center}
\caption{The number of GMRES iterations required to solve the Laplace preconditioned system in the three dimensional test case.}
\label{fig:test2_laplace_GMRES}
\end{figure}

The required number of GMRES iterations for different parameter values is presented in Fig. \ref{fig:test2_laplace_GMRES}. Based on these results, the number of iterations behaves qualitatively as predicted by the theory. The number of iterations is dependent both on $\kappa^2$ and $\sigma$. When $\sigma$ grows, the presented theory predicts that also the distance from the origin to FOV increases. This can be seen as the changing slope of the $\kappa^2$ to number of iterations lines in the Fig. \ref{fig:test2_laplace_GMRES}. 

Next, we replace the exact Laplace preconditioner with a multigrid based preconditioner. The interesting question here is the dependency of the number of GMRES iterations on the number of $V$-cycles. The required number of GMRES iterations for different parameter values is presented in Fig. \ref{fig:test2_mg_GMRES}. Based on these results one can clearly observe that the required number of iterations is strongly dependent on the number of multigrid $V$-cycles. Again, the inclusion of the origin into the FOV is not observed in the required number of GMRES iterations.

\begin{figure}
\begin{center}
\hfill
\includegraphics[scale=0.5]{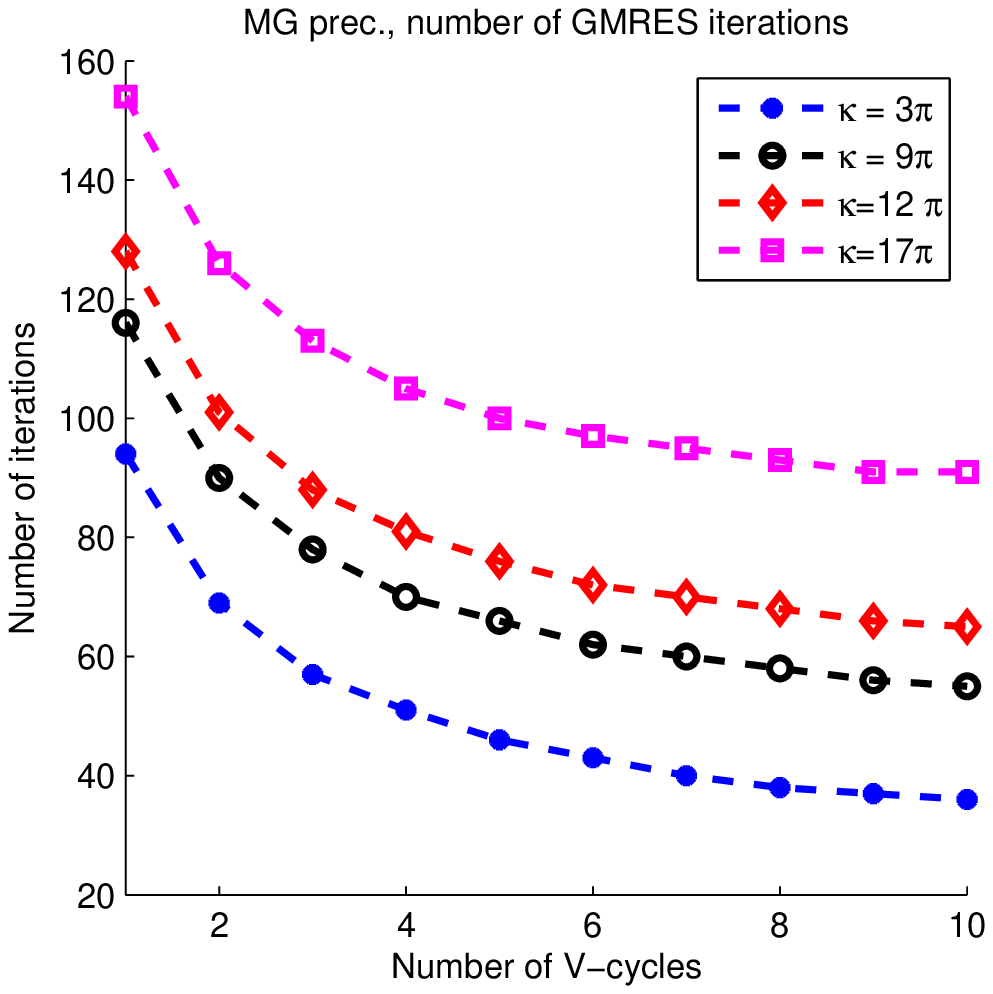} \hfill 
\includegraphics[scale=0.5]{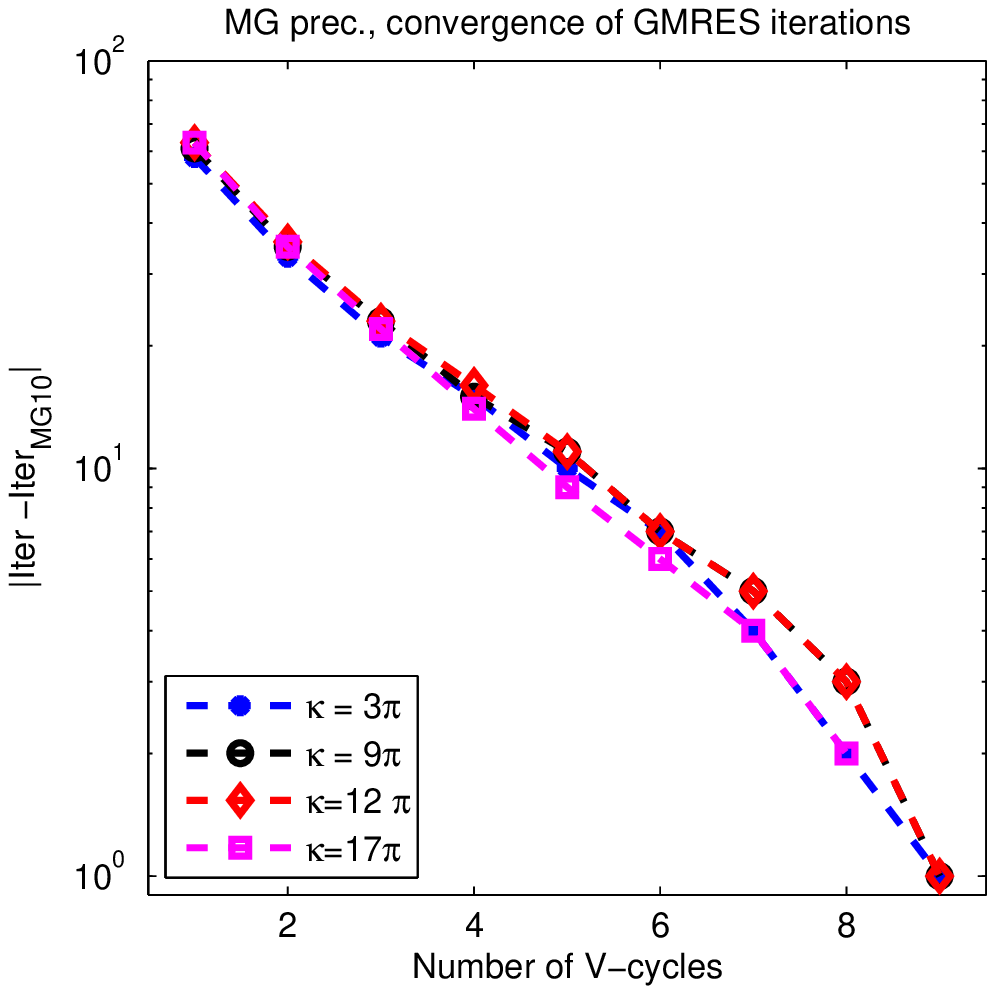}
\hfill
\end{center}
\caption{ The left figure shows the number of GMRES iterations required to solve the three dimensional test problem. The right figure shows the convergence of the number of iterations by comparing the number of iterations for varying number of $V$-cycles to the MG preconditioner using ten $V$-cycles.}
\label{fig:test2_mg_GMRES}
\end{figure}

Finally, we consider the two-level preconditioner. The mesh levels from one to four are used as the coarse grid. The same procedure as for the exact Laplace preconditioner is used to solve the Poisson problem on the fine grid. The number of GMRES iterations for different values of $\kappa$ and different coarse grid levels are presented in Table \ref{table:test2_2L_GMRES}. From the number of iterations one can qualitatively verify the result given in Section 6. Namely, large values of $\kappa$ require more dense coarse meshes before the number of iterations converges to a limit value. In the present case, this happened only for the parameter value $\kappa=4 \pi$. Probably, even a more dense computational mesh would be required to resolve the solution for $\kappa=10\pi$.

\begin{table}
\begin{center}
\begin{tabular}{|l|l|l|l|l|}
 	 & \multicolumn{4}{|c|}{Coarse grid levels}\\
\hline
$\kappa$ & 1 & 2 & 3 & 4 
\\\hline	
$4 \pi$  & $52$   & $37$   & $28$   & $25$ \\
$6 \pi$  & $197$  & $131$  & $73$   & $37$ \\
$10\pi$  & n.c.   & n.c    & n.c.   & $130$ \\
\end{tabular}
\caption{ Number of GMRES iterations required to solve the problem using the two-level preconditioner. The computational mesh has approximately $750 \cdot 10^3$ degrees of freedom and $4.5 \cdot 10^6$ tetrahedral elements. The iteration is deemed not to converge, if more than two hundred iterations are required.}
\label{table:test2_2L_GMRES}
\end{center}
\end{table}

\bibliographystyle{spmpsci}      % mathematics and physical sciences


\bibliography{Sifted_PREC}

\end{document}